\documentclass[11pt,a4paper,twoside]{article}
\usepackage[UKenglish]{babel}
\usepackage[T1]{fontenc}
\usepackage[utf8x]{inputenc}
\usepackage{latexsym,amsfonts,amsmath,amsthm,amssymb, mathrsfs}
\usepackage{fullpage}
\usepackage{xcolor,bm,bbm}
\usepackage{paralist}
\usepackage{todonotes}
\usepackage{fourier}
\usepackage{amssymb}

\usepackage{tikz}
\usepackage{pgfplots}
\usepackage[labelformat=empty]{caption}
\usetikzlibrary{patterns}

\DeclareMathOperator{\N}{\mathbb N}
\DeclareMathOperator{\eps }{ \varepsilon }

\newtheorem{thm}{Theorem}[section]

\newtheorem{example}[thm]{Example}
\newtheorem{lemma}[thm]{Lemma}
\newtheorem{thmalpha}{Theorem}

{
\theoremstyle{definition}

}

{
\theoremstyle{definition}
}

\newtheorem{problem}{Problem}

\allowdisplaybreaks
\begin{document}

\title{Arithmetic sensitivity of cumulant growth in
lacunary sums: transcendental versus algebraic ratio limits}
\author{Christoph Aistleitner\footnote{Graz University of Technology, Institute of Analysis and Number Theory, Steyrergasse 30, 8010 Graz, Austria. Email: aistleitner@math.tugraz.at}, Zakhar Kabluchko\footnote{University of M\"unster, Institute for Mathematical Stochastics, Orleans-Ring 10, 48149 M\"unster, Germany. Email: kabluchk@uni-muenster.de}, Joscha Prochno\footnote{University of Passau, Faculty of Computer Science and Mathematics, Dr.-Hans-Kapfinger-Stra{\ss}e 30, 94032 Passau, Germany. Email: joscha.prochno@uni-passau.de}}
\date{}

\maketitle





\begin{abstract}
We study the asymptotic behavior of cumulants of lacunary trigonometric sums $S_n(\omega) := \sum_{k=1}^n  \cos (2 \pi a_k \omega)$, $\omega\in[0,1]$, and show that cumulant growth is highly sensitive to the arithmetic structure of the sequence $(a_k)_{k \geq 1}$ of positive integers. In particular, if $\lim_{k \to \infty} a_{k+1}/a_k = \eta > 1$ for some transcendental number $\eta$, we prove that for every $m\in \mathbb N$  the $m$-th cumulant of $S_n$ is asymptotically equivalent to  the $m$-th cumulant of the ``independent model'' $\widetilde{S}_n := \sum_{k=1}^n  \cos (2 \pi a_k U_k)$, where $U_1, U_2, \dots$ are independent random variables having uniform distribution on $[0,1]$.  In particular, the order of growth of the cumulants as $n \to \infty$ is linear in this case. We also show that the transcendence condition for $\lim_{k \to \infty} a_{k+1}/a_k$ is in general necessary: when the ratio limit $\eta$ is algebraic,
the cumulants of $S_n$ may have a different asymptotic order from those of $\widetilde{S}_n$.   For instance, for $a_k = 2^k+1$ (with $\eta = 2$), the sixth cumulant of $S_n$ grows quadratically in $n$. In contrast, for $a_k = 2^k$ (again $\eta = 2$) or when  $(a_k)_{k \geq 1}$ is the Fibonacci sequence (with $\eta = (1+\sqrt 5)/2$), the $m$-th cumulant of $S_n$ grows linearly as $n\to\infty$, but with a growth rate that differs from the one of the independent model $\widetilde{S}_n$. Overall, our results show that the asymptotic behavior of the cumulants of lacunary trigonometric sums depends on arithmetic effects in a very delicate way. This is particularly remarkable since many other probabilistic limit theorems, such as the Central Limit Theorem, hold for lacunary trigonometric sums in a universal way without any such sensitivity towards arithmetic effects. 
\bigskip

\noindent\textbf{Subjectclass:} Primary 42A55; Secondary 11D45, 11K06, 11K70, 42A70, 60F05, 60F10
\end{abstract}
	 
\maketitle

\section{Introduction and main results}

\subsection{Introduction}

Let $(a_k)_{k \geq 1}$ be a sequence of positive integers satisfying the Hadamard gap condition
 	\[
		\frac{a_{k+1}}{a_k} \geq q, \qquad k \geq 1,
	\]
for some number $q>1$.
We are interested in the lacunary trigonometric sums
$$
S_n(\omega) := \sum_{k=1}^n  \cos (2 \pi a_k \omega),
\qquad
\omega\in[0,1],
\quad
n\in \N,
$$
considered as random variables on the probability space $[0,1]$,  endowed with the Borel $\sigma$-algebra and Lebesgue measure.  It is well-known that lacunary trigonometric sums or, more generally, lacunary sums of dilated periodic functions, exhibit many properties typically satisfied by sums of independent and identically distributed (i.i.d.)\ random variables. In our setting, the random variables $X_k(\omega) := \cos(2\pi a_k \omega)$, $k\in\mathbb N$, are identically distributed and uncorrelated (assuming the $a_k$, $k\in\mathbb N$, are distinct); but the $X_k$'s are not independent and not even stationary. Classical works of Kac \cite{Kac_biggaps}, Salem and Zygmund \cite{SZ1947} or Erd\"{o}s and G{\'a}l \cite{EG1955} that appeared in the mid-20th century have shown that under the Hadamard gap condition the suitably normalized partial sums $S_n$, despite the dependence of summands, asymptotically behave like sums of i.i.d.\ random variables, satisfying a central limit theorem (CLT) and a law of the iterated logarithm (LIL) with the same normalization as in the i.i.d.\ case; see~\cite{aistleitner2024lacunarysequencesanalysisprobability} for more results and references.  Only recently it has been shown by Aistleitner, Gantert, Kabluchko, Prochno and Ramanan in \cite{agkpr} (see also \cite{FJP2022}) that this benign (and maybe expected) behavior rather surprisingly may break down when large deviation probabilities $\mathbb P[S_n/n >u]$, $u\in (0,1)$,  are considered. Here the arithmetic structure of the  sequence $(a_k)_{k\geq 1}$ suddenly plays a crucial role and strongly influences the limiting behavior such that, depending on the sequence,  the analogue of Cram\'er's theorem may hold with the same rate function as in the i.i.d.\ case, or it may hold with a different rate function, or the Cram\'er theorem may break down completely.
The ultimate goal is thus to deeply understand the delicate interplay of arithmetic and analytic effects on the probabilistic behavior of lacunary sums.
In this paper our focus will be on the cumulants of $S_n$, which are known to be fundamental characteristics of a probability distribution, describing the shape, its mean, variance, skewness, and various other properties.
 
We recall that the $m$-th cumulant of a random variable $X$ whose moment generating function $m_X(t):= \mathbb E [e^{tX}]$ is finite for $t\in \mathbb R$  can be defined by
$$
\kappa_m(X) =  \frac{d^m}{dt^m} \log m_X(t) \Big |_{t=0}, \qquad m\in \N. 
$$
Note that $\kappa_1(X) = \mathbb E X$, $\kappa_2(X)$ is the variance of $X$, while $\kappa_3(X)$ is the third central moment. In general, $\kappa_m(X)$ can be expressed through the first $m$ moments of $X$; see, e.g., \cite[Section~3]{peccati_taqqu_book_wiener_chaos}. A basic property of cumulants  is their additivity, i.e.,  $\kappa_m(X+Y) = \kappa_m(X) + \kappa_m(Y)$ for independent random variables $X$ and $Y$.



We shall compare the cumulants of the lacunary trigonometric sums $S_n$, $n\in\mathbb N$, with the cumulants of the ``independent model'' $\widetilde{S}_n := \sum_{k=1}^n  \cos (2 \pi a_k U_k)$, $n\in\mathbb N$, where $U_1, U_2, \dots$ are independent random variables having uniform distribution on $[0,1]$; in what follows, we write $U \sim\mathrm{Unif}[0,1]$ for a random variable having uniform distribution on $[0,1]$.  The common distribution of the i.i.d.\  random variables $\widetilde{X}_k:= \cos(2\pi a_k U_k)$, $k\in\mathbb N$, is the arcsine law on the interval $(-1,1)$, which has Lebesgue density
\[
f(x) = \frac{1}{\pi\sqrt{1-x^2}}, \qquad |x|<1.
\]
The moment generating function of the arcsine distribution is the modified Bessel function of the first kind $I_0$, which has the series expansion
\begin{equation} \label{bessel_series}
I_0(t) 
= \sum_{j=0}^\infty \frac{(t/2)^{2j}}{(j!)^{2}}, \quad t\in\mathbb R.
\end{equation}
More precisely, we have
\begin{equation} \label{arcsine_bessel}
m_{\widetilde{X}_k}(t)
=
\mathbb E [e^{t \widetilde{X}_k}]
=
\frac 1 {\pi} \int_{-1}^{1} \frac{e^{tx} dx}{\sqrt{1-x^2}}
=
\frac 1{\pi} \int_0^\pi e^{t \cos u} du = I_0(t).
\end{equation}
The cumulants of $\widetilde{X}_k$ are thus given by
$$
\widetilde{\kappa}_m :=  \kappa_m(\widetilde{X}_k) = \frac{d^m}{d t^m} \log I_0(t) \Big|_{t=0}, \qquad m\in \N, \quad k\in \N.
$$
Note that the cumulants with odd index $m$ are all zero since the function $\log I_0(t)$ is even, c.f.\ the series expansion~\eqref{bessel_series}. 
For the first few cumulants with even indices, we obtain the values
\begin{equation}  \label{cumulants_indep}
\widetilde{\kappa}_2 = \frac{1}{2}, \qquad \widetilde{\kappa}_4 = -\frac{3}{8}, \qquad \widetilde{\kappa}_6  = \frac{5}{4}, \qquad \widetilde{\kappa}_8 = -\frac{1155}{128},
\qquad
\widetilde{\kappa}_{10} = \frac{3591}{32}.
\end{equation}
The sequence $(2^{2j} \widetilde{\kappa}_{2j})_{j\geq 1}$ appears as entry A352284  in the On-Line Encyclopedia of Integer Sequences (OEIS)~\cite{sloane}; see also A352313 for a version without signs.  Since cumulants are additive for independent random variables,  the $m$-th cumulant of $\widetilde{S}_n=\widetilde{X}_1 + \ldots + \widetilde{X}_n$ equals $n \widetilde{\kappa}_m$, i.e., $\kappa_m(\widetilde{S}_n)=n \widetilde{\kappa}_m$. 

\subsection{Sequences with transcendental ratio limit}

The main purpose of the present paper is to show that the cumulants of $S_n$ depend on the arithmetic structure of the sequence $(a_k)_{k \geq 1}$ in an extremely delicate way. In Theorem \ref{th1} below we prove that when $a_{k+1} / a_k$ converges towards a transcendental number $\eta>1$, then the cumulants of $S_n$ behave asymptotically in the same way as those of $\widetilde{S}_n$; recall that a transcendental number is a real or complex number that is not algebraic, i.e., not the root of a non-zero polynomial with integer coefficients. Afterwards, in Theorems \ref{th2} and \ref{theo:fibonacci_general} presented in the following subsections, we show that this condition is in a sense optimal: we exhibit examples of sequences for which $a_{k+1}/a_k$ converges towards an integer, or towards an algebraic number, and for which the cumulants of $S_n$ not only fail to satisfy $\kappa_m(S_n) \sim \kappa_m(\widetilde{S}_n)$ as $n\to\infty$, but for which the cumulants of $S_n$ actually turn out to be of a completely ``wrong'' asymptotic order in comparison with those of $\widetilde{S}_n$. 

In the first theorem, dealing with the case of a transcendental ration limit, we use the standard Bachmann--Landau notation to denote the order of approximation.

\begin{thmalpha}[Sequences with transcendental ratio limit] \label{th1}
Let $(a_k)_{k \geq 1}$ be an increasing sequence of integers such that
\begin{equation} \label{eta_conv}
\lim_{k \to \infty} \frac{a_{k+1}}{a_k} = \eta > 1
\end{equation}
for some transcendental number $\eta$. Then, for all integers $m \geq 1$, we have
$$
\kappa_{m}(S_n) - n \widetilde{\kappa}_m = \mathcal{O}(1) ,\qquad \text{as $n \to \infty$}.
$$
In particular, $\kappa_{m}(S_n)/n \to \widetilde{\kappa}_m$, as $n\to\infty$.  
\end{thmalpha}

\subsection{An example with non-linear growth of cumulants}

A famous example of a lacunary sum which exhibits ``irregular'' behavior is the system $(f(a_k \omega))_{k \geq 1}$, $\omega\in [0,1]$,  where
$$
f(x) = \cos(2 \pi x) + \cos(4 \pi x) \qquad \text{and} \qquad a_k = 2^k + 1, \quad k \geq 1.
$$
The example was first attributed to Erd\H os and Fortet in papers of Salem and Zygmund \cite{salz} and of Kac \cite{kac}, and is now known as the Erd\H os--Fortet example (cf.\ also \cite{clb,fm}).
 The corresponding lacunary sum is
\begin{eqnarray*}
\sum_{k=1}^n f(a_k \omega) & = & \sum_{k=1}^n \left( \cos\big(2 \pi (2^k+1) \omega\big) + \cos\big(4 \pi  (2^k+1) \omega\big) \right), \qquad \omega \in [0,1].
\end{eqnarray*}
Noting that the term $\cos(4 \pi  (2^k+1) \omega)$ for an index $k$ and the term $\cos(2 \pi (2^{k+1}+1) \omega)$ for an index $k+1$ can be combined by means of the trigonometric sum-to-product identity $\cos(x)+\cos(y) = 2 \cos(\tfrac{x+y}{2})\cos(\tfrac{x-y}{2})$, we obtain
$$
\cos\big(4 \pi  (2^k+1) \omega\big) + \cos\big(2 \pi (2^{k+1}+1) \omega\big) = 2 \cos(\pi \omega) \cos \left(4 \pi \left(2^k + \frac{3}{4} \right) \right),
$$
and so the sum can be re-written as
$$
\sum_{k=1}^n f(a_k \omega) = \cos\big(6 \pi \omega\big) + \cos\big(4 \pi (2^n+1) \omega\big) + 2 \cos(\pi \omega) \sum_{k=1}^{n-1} \cos \left(4 \pi \left(2^k + \frac{3}{4} \right) \right).
$$
Here $\sum_{k=1}^{n-1} \cos \left(4 \pi \left(2^k + \frac{3}{4} \right) \right)$ is a ``pure'' trigonometric sum, in the sense that each summand is a single cosine term (as opposed to a lacunary sum of general  trigonometric polynomials involving several terms), and behaves in a ``random'' way, so that the system $(f(a_k \omega))_{k \geq 1}$ behaves in many regards like a system of i.i.d.\ random variables with variance $1/2$, all of which  have been multiplied with the (independent) factor $2 \cos(\pi \omega)$. Thus, for example, the law of the iterated logarithm for this system takes the form
$$
\limsup_{n \to \infty}  \frac{\left| \sum_{k=1}^n f(a_k \omega) \right|}{\sqrt{n \log \log n}} = |2 \cos(\pi \omega)|\qquad \textup{almost everywhere},
$$
and the central limit theorem holds in a similar form with a so-called ``variance mixture Gaussian'' as the limit distribution. Note, however, that all of this ``irregular'' behavior is only observed since $f$ is a trigonometric polynomial with more than one term; for the pure trigonometric sum
$$
\sum_{k=1}^n \cos\big(2 \pi  (2^k +1) \omega\big),
$$
the CLT and LIL hold in their universal form, in the same way as they hold for every other lacunary sequence. Thus, as far as pure lacunary sums are concerned, the CLT and LIL cannot ``detect'' the difference in the arithmetic structure of the sequences $(2^k)_{k \geq 1}$ and $(2^k+1)_{k \geq 1}$, say, or a sequence $(a_k)_{k \geq 1}$ with $\frac{a_{k+1}}{a_k}$ tending towards a transcendantal number as in the statement of Theorem \ref{th1}. Our next theorem shows that cumulants on the other hand can detect this difference, even in the case of pure trigonometric sums.


\begin{thmalpha}[Non-linear growth of cumulants]\label{th2}
Let $a_k = 2^k+1$ for $k \geq 1$, and $S_n(\omega) = \sum_{k=1}^n \cos(2 \pi a_k \omega)$, $\omega\in[0,1]$. Let $\kappa_m(S_n)$ be the $m$-th cumulant of $S_n$. Then we have $\kappa_1(S_n) = \kappa_3(S_n) = \kappa_5(S_n) = \dots = 0$ for $n \geq 1$, as well as
\begin{eqnarray*}
\kappa_2(S_n) & = & \frac{n}{2} \qquad \text{for $n \geq 1$},\\
\kappa_4(S_n) & = & \frac{-3n + 28}{8} \qquad \text{for $n \geq 4$},\\
\kappa_6(S_n) & = & \frac{45n^2 + 380n -1875}{16} \qquad \text{for $n \geq 7$.}
\end{eqnarray*}
\end{thmalpha}

The significant point in Theorem \ref{th2} is that the cumulants of $S_n$ not only fail to coincide with those of $\widetilde{S}_n$, which are given by $n \widetilde{\kappa}_m$ for the constants from \eqref{cumulants_indep}. Even more drastically, while the cumulants of $\widetilde{S}_n$ grow all linearly (as a consequence of the fact that $\widetilde{S}_n$ is a sum of $n$ i.i.d.\ variables), the cumulants of $S_n$ fail to do so and have (in general) a different asymptotic order. This effect becomes first visible for the $6$-th cumulant, but could also be observed for cumulants of higher order; however, to keep the exposition short, 
we have refrained from calculating cumulants of higher order than $6$. 
In statistical physics, cumulants are often used as  examples of ``extensive quantities'', i.e., those that scale with system size: for systems with short-range (summable) correlations, the $m$-th cumulant of a sum of local observables grows linearly with the size $n$ of the system.  For systems with non-summable pair correlations, cumulants may grow non-linearly, and this effect is usually already visible in the second cumulant. Theorem~\ref{th2} provides an example of an \textit{uncorrelated} system for which the $6$-th cumulant loses its extensivity.

\subsection{Recursive sequences with dominant root condition}
In this section, we shall study the cumulants of lacunary sums in the setting where $(a_k)_{k \geq 1}$ is a recursive sequence whose characteristic polynomial has a dominant real root. In this setting, $a_{k+1}/a_k$ converges to the dominant real root, which is an algebraic number.

The simplest special case of this setting is the famous Fibonacci sequence $(F_k)_{k\in\N}$ defined via the recurrence relation
	\[
F_k=F_{k-1}+F_{k-2} \text{ for all } k\geq 3,  \quad \text{ with } F_1=F_2 = 1.
	\]
By Binet's formula, $F_k = (\varphi^k - \psi^k)/\sqrt 5$, $k\in \N$,  where $\varphi = (1+\sqrt 5)/2$ is the golden ratio  and $\psi = (1-\sqrt 5)/2$. In particular, as it had already been observed by Johannes Kepler, the ratio limit is
	\[
	  \eta = \lim_{k\to\infty} \frac{F_{k+1}}{F_k} = \frac{1+\sqrt{5}}{2} = \varphi >1.
	\]
Note that $\varphi$ is an algebraic number since it solves  $\varphi^2-\varphi-1=0$.  We will show that the cumulants for lacunary sums $S_n$, $n\in\mathbb N$, involving the Fibonacci sequence grow linearly in $n$, but the asymptotic slope is not the same  as in the independent case $\widetilde{S}_n$.

More generally, we consider lacunary sums $S_n(\omega) = \sum_{k=1}^n  \cos (2 \pi a_k \omega)$, $\omega\in[0,1]$, where $a_1,a_2,\ldots$ are positive integers admitting the representation
\begin{equation}\label{eq:a_n_recursive_binet_representation}
a_k = c_1 \lambda_1^k + \ldots+ c_d \lambda_d^k, \qquad k\in \mathbb N,
\end{equation}
where $d\in \N$ and
\begin{itemize}
\item[(i)] $\lambda_1,\ldots, \lambda_d\in \mathbb C$ are roots of some irreducible degree $d$ polynomial $P(z) = \sum_{j=0}^d r_j z^j$ with integer coefficients $r_0,\ldots, r_d \in \mathbb Z$;
\item[(ii)] $c_1,\ldots, c_d$ are complex numbers;
\item[(iii)] the following \textit{dominant root condition} holds:
\begin{equation}\label{eq:dominant_root_condition}
\lambda _1 \text{ is real }, \qquad \lambda_1 >1, \qquad \lambda_1 > \max\big\{|\lambda_2|,\ldots, |\lambda_d|\big\} =: \rho, \qquad c_1 \neq 0.
\end{equation}
\end{itemize}

Clearly, $a_{k+1}/a_k \to \lambda_1 =: \eta$.
Algebraic numbers $\lambda_1$ whose Galois conjugates satisfy~\eqref{eq:dominant_root_condition} are called \textit{Perron numbers}. It follows from~\eqref{eq:a_n_recursive_binet_representation} and~(i) that the sequence $(a_k)_{k\geq 1}$ satisfies the linear recursion relation $r_d a_{k+d} = -\sum_{j=0}^{d-1} r_j a_{k+j}$, $k\in \N$.

\begin{example}
The Fibonacci numbers $F_k$ satisfy the above conditions, as do the  Lucas numbers $L_k := \varphi^k + \psi^k$.
\end{example}

\begin{example}
The sequence $a_k = c \eta^k$, where $c\in \N$ and $\eta \in \{2,3,\ldots\}$, satisfies the above conditions with $d=1$.
\end{example}

\begin{thmalpha}[Recursive sequences with dominant root condition]\label{theo:fibonacci_general}
Let $(a_k)_{k\geq 1}$ be a sequence of positive integers satisfying~\eqref{eq:a_n_recursive_binet_representation} and (i), (ii), (iii).  Then, for every $m\in \mathbb N$, the sequence $n\mapsto \kappa_m(S_n)$ becomes eventually linear. More precisely, for a  sufficiently large integer $n_1(m)$, we have
$$
\kappa_m(S_n) = 2^{-m} (w_m n + b_m), \qquad n>n_1(m),
$$
where $w_m\in \mathbb Z$ and $b_m\in \mathbb Z$ are integers depending  only on $m$ and $c_1,\ldots, c_d, \lambda_1,\ldots,\lambda_d$.
\end{thmalpha}
A proof of Theorem~\ref{theo:fibonacci_general} will be given in Sections~\ref{sec:combinatorial_formula_cumulants} and~\ref{sec:proof_fibonacci_case}. Section~\ref{sec:combinatorial_formula_cumulants} contains a general combinatorial formula for the cumulants of $S_n$ (valid for every sequence $(a_k)_{k\geq 1}$ of natural numbers) that may be of independent interest. The proof of Theorem~\ref{theo:fibonacci_general} is constructive in the sense that it gives an algorithm to compute $w_m$, $b_m$, and $n_1(m)$ for a given $m\in \N$.

\begin{example}
In the Fibonacci case $a_k = F_k$,  brute-force calculations give, for sufficiently large $n\in\mathbb N$,
\[
\kappa_1(S_n) = 0,
\quad
\kappa_2(S_n) = \frac{n}{2}+1,
\quad
\kappa_3(S_n) = \frac{3}{2}\,n,
\quad
\kappa_4(S_n) = \frac{45n - 106}{8},
\quad
\kappa_5(S_n) = 20 n - \frac{2145}{16}.
\]
For $a_k= 2^k$, formulas for $\kappa_m(S_n)$ with $m=1,\ldots, 7$ (and sufficiently large $n$) were obtained in~\cite[p.~550]{agkpr}, but no proof of eventual linearity for general $m$ has been given there. In both cases, the linear slope of $\kappa_3(S_n)$ is non-zero, which should be contrasted to the value $\widetilde{\kappa}_3 = 0$ corresponding to the independent model.
\end{example}

\subsection{Open problems}

We conclude this section with some open problems, and suggestions for further research.

\begin{problem}
As our Theorems \ref{th2} and \ref{theo:fibonacci_general} show, when $\lim_{k \to \infty} \frac{a_{k+1}}{a_k} =: \eta$ is allowed to be an algebraic number, then the asymptotic order of the cumulants of the lacunary trigonometric sum (as $n \to \infty$) can differ significantly from the asymptotic behavior of the corresponding independent model. However, we believe that for any given $\eta>1$ (including integers, rationals and algebraic numbers) one can construct a lacunary sequence $(a_k)_{k \geq 1}$ such that $\lim_{k \to \infty} \frac{a_{k+1}}{a_k} = \eta$, and such that the asymptotic behavior of the cumulants of the lacunary trigonometric sum coincides with the behavior of the independent model. The existence of such a sequence can probably most easily be shown in a randomized way (random perturbation of a deterministic sequence with asymptotic ratio limit $\eta$), in the spirit of Theorem D of \cite{agkpr}; an explicit deterministic construction is probably much more difficult.
\end{problem}

\begin{problem}
Our Theorem  \ref{theo:fibonacci_general} gives a general framework for the asymptotic behavior of lacunary trigonometric sums for sequences $(a_k)_{k \geq 1}$ that satisfy a linear recurrence relation that is associated with a Perron number. It would be interesting to also study ``simple'' sequences for which $\lim_{k \to \infty} \frac{a_{k+1}}{a_k} = \eta > 1$ is algebraic, but not a Perron number. A natural and particularly  interesting example are sequences of the type $( \lfloor \eta^k \rfloor)_{k \geq 1}$, where $\eta>1$ is algebraic but not Perron. We conjecture that in this particular case, the cumulants of the lacunary trigonometric sum grow linearly as $n \to \infty$, but with a growth factor which is different from the one in the corresponding independent model. 
\end{problem}

\begin{problem}
We restate a problem from \cite{agkpr}. That paper studied large deviations principles (LDPs) for lacunary trigonometric sums, and also observed a high degree of ``arithmetic sensitivity'' in the precise way how such lacunary sums satisfy an LDP. It was conjectured in \cite{agkpr} that whenever $\lim_{k \to \infty} \frac{a_{k+1}}{a_k} = \eta > 1$ for some transcendental $\eta$, then the lacunary trigonometric sum satisfies an LDP with exactly the same rate function as for the corresponding independent model. Such a result would be an analogue of Theorem \ref{th1} in the present paper. It is possible that some ideas from the present paper can be used to approach this problem, but on a technical level there are substantial differences between the setup of the present paper and the LDP setup, and additional ideas would be necessary to settle the problem. In this context, we believe that a version of Theorem \ref{theo:fibonacci_general} of the present paper should also carry over to the LDP setup, in the sense that for lacunary sequences $(a_k)_{k \geq 1}$ generated by a linear recurrence relation as in Theorem \ref{theo:fibonacci_general}, the trigonometric sums should satisfy a large deviations principle, but with a rate function which is in general different from the one for the corresponding independent model.
\end{problem}

\section{Preparations for the proof of Theorem \ref{th1}} \label{sec_2}

We start with a few technical ingredients that shall be used later in the proof of the first main result, Theorem \ref{th1}.

\begin{lemma} \label{lemma_poly}
For given $m, h \in\mathbb N$, let $\mathcal{P}:=\mathcal{P}_{m,h}$ denote the class of functions on $\mathbb R^h$ of the form
$$
p(y_1, y_2, y_3, \dots, y_h) = b_1 y_1^{e_1} + b_2 y_2^{e_2} + b_3 y_3^{e_3} + \dots+ b_h y_h^{e_h},
$$
where $0 = e_1 <  e_2 <  \dots < e_h$ are integers, and $b_1, \dots, b_h\in\mathbb Z$ are such that $|b_i| \leq m$ for $1 \leq i \leq h$, and where additionally we require that $(b_1, \dots, b_h) \neq (0,\dots,0)$. Let $\eta > 1$ be transcendental. Then there exist constants $\delta>0$ and $\varepsilon>0$, depending only on $\eta,m,h$, such that
\begin{equation} \label{bounded_away}
|p(y_1, y_2, y_3, \dots, y_h)| \geq \varepsilon
\end{equation}
for all $p \in \mathcal{P}$, and for all $(y_1,y_2, \dots, y_h) \in \mathbb{R}^{h}$ which have the property that
\begin{equation} \label{y_assumption}
\frac{y_j^{e_j}}{y_i^{e_i}} \in \left[ \left( \eta - \delta \right)^{e_j-e_i}, \left( \eta + \delta \right)^{e_j-e_i} \right] \qquad \text{for all $1 \leq i < j \leq h$}.
\end{equation}
\end{lemma}

\begin{proof}
We assume that $\eta>1$ is transcendental, and that $m \geq 1$ is given and fixed. We will argue by induction over $h$.
\vskip 1mm
\noindent \emph{Step 1.} For $h=1$ the conclusion is trivial, since $\mathcal{P}_{m,1}$ is merely the class of constant functions. Indeed, in \eqref{bounded_away} we have $|p(y_1)| = |b_1|$ for any $y_1\in\mathbb R$ and with a non-zero integer $b_1$; note that $e_1=0$ by assumption. So the modulus of functions is always bounded below by $1$.
\vskip 1mm
\noindent \emph{Step 2.} Let $h\in \mathbb N$, $h\geq 2$. Assume that the conclusion of the lemma is true for $k=1, \dots, h-1$. We need to establish its validity for $h$. We can assume that $b_h \neq 0$, since otherwise the induction hypothesis can be directly applied.

So we can assume that for $k \in\{1, \dots, h-1\}$ there always exist suitable constants $\delta_k>0$ and $\varepsilon_k>0$ for which the conclusion of the lemma holds (we suppress the dependence of $\delta_k$ and $\varepsilon_k$ on $\eta$ and $m$, since they are assumed to be fixed). We can assume w.l.o.g.\ that all constants $\delta_1, \delta_2, \dots, \delta_{h-1}$ have been chosen smaller than $(\eta-1)/2$, and that $\delta_h$ will also be chosen smaller than $(\eta-1)/2$, so that we have
\begin{equation} \label{etadelta}
\eta - \delta_k \geq \frac{\eta}{2} + \frac{1}{2} > 1, \qquad \text{for all $k=1, \dots, h$}.
\end{equation}
Note that in view of \eqref{y_assumption} this implies for the considered variables the lower bound $1$, i.e.,
  \[
    \forall j\in\{2,\dots,h\}: \quad y_j \geq 1;
  \]
in particular, all variables $y_2,\dots,y_h$ that we consider are positive. Let $c := c(h) \in\mathbb N$ be an integer large enough so that
\begin{equation} \label{choice_of_c}
\left( \frac{\eta}{2} + \frac{1}{2} \right)^c > 2 m h  \left(\max_{1 \leq k \leq h-1} \frac{1}{\varepsilon_k}\right).
\end{equation}
We now distinguish two cases:\\

\begin{itemize}
 \item{Case 1}: Among the exponents $e_1 < \dots < e_h$ there exists an index $r$ such that $e_{r+1} - e_r \geq c$. Assume that $r$ is the largest such index, i.e., that $e_{r+2} - e_{r+1} < c$, $e_{r+3} - e_{r+2} < c$, etc.\ To cover Case 1 we can pick $\delta_h = \min_{1 \leq k \leq h-1} \delta_k$ (but we may need to reduce $\delta_h$ in the second step, in order to cover Case 2 as well; thus during the Case 1 analysis we will work with the assumption that $\delta_h \leq \min_{1 \leq k \leq h-1} \delta_k$). Assume that we are given $(y_1,y_2, \dots, y_h) \in \mathbb{R}^{h}$ which satisfy
 \begin{equation} \label{y_assumption_mod}
\frac{y_j^{e_j}}{y_i^{e_i}} \in \left[ \left( \eta - \delta_h \right)^{e_j-e_i}, \left( \eta + \delta_h \right)^{e_j-e_i} \right] \qquad \text{for all $1 \leq i < j \leq h$}.
\end{equation}
 Then we have
 \begin{eqnarray}
 & & \left| b_1 y_1^{e_1} +  \dots + b_h y_h^{e_h} \right| \nonumber\\
 & \geq & - \left| b_1 y_1^{e_1} + \dots + b_r y_r^{e_r} \right| + \left| b_{r+1} y_{r+1}^{e_{r+1}} + \dots + b_h y_h^{e_h} \right| \nonumber\\
 & \geq & - m r y_r^{e_r} + y_{r+1}^{e_{r+1}}  \left| b_{r+1}  \frac{y_{r+1}^{e_{r+1}}}{y_{r+1}^{e_{r+1}}}  +  b_{r+2} \frac{y_{r+2}^{e_{r+2}}}{y_{r+1}^{e_{r+1}}} + \dots + b_h \frac{y_h^{e_h}}{y_{r+1}^{e_{r+1}}} \right| . \label{continue}
 \end{eqnarray}
By setting
$$
c_j := b_{r+j}, \qquad f_j := e_{r+j} - e_{r+1}, \qquad z_j :=  \left( \frac{y_{r+j}^{e_{r+j}}}{y_{r+1}^{e_r+1}}  \right)^{1/f_j}, \qquad \text{for $1 \leq j \leq h-r$,}
$$
we can write
\begin{equation} \label{continue_2}
\left|b_{r+1}  \frac{y_{r+1}^{e_{r+1}}}{y_{r+1}^{e_{r+1}}}  +  b_{r+2} \frac{y_{r+2}^{e_{r+2}}}{y_{r+1}^{e_{r+1}}} + \dots + b_h \frac{y_h^{e_h}}{y_{r+1}^{e_{r+1}}}  \right| = \left| c_{1}  z_{1}^{f_{1}}  +  c_2 z_{2}^{f_2} + \dots + c_{h-r}  z_{h-r} ^{f_{h-r}} \right|.
\end{equation}
Here $0 = f_1 < f_2 < \dots < f_{h-r}$ are integers, we have $|c_j| \leq m$ for all $j$, and we also have $(c_1, \dots,  c_{h-r}) \neq (0, \dots, 0)$ since we assumed $0 \neq b_h = c_{h-r}$. Furthermore, as a consequence of \eqref{y_assumption_mod} we have
$$
\frac{z_j^{f_j}}{z_i^{f_i}} = \frac{y_{r+j}^{e_{r+j}}}{y_{r+i}^{e_{r+i}}} \in \left[ \left( \eta - \delta_h \right)^{e_{r+j}-e_{r+i}}, \left( \eta + \delta_h \right)^{e_{r+j}-e_{r+i}} \right]  \qquad \text{for $1 \leq i < j \leq h-r$},
$$
where we note that
\begin{eqnarray*}
 \left[ \left( \eta - \delta_h \right)^{e_{r+j}-e_{r+i}}, \left( \eta + \delta_h \right)^{e_{r+j}-e_{r+i}} \right]  & = & \left[ \left( \eta - \delta_h \right)^{f_j-f_i}, \left( \eta + \delta_h \right)^{f_j-f_i} \right] \\
 & \subset & \left[ \left( \eta - \delta_{h-r} \right)^{f_j-f_i}, \left( \eta + \delta_{h-r} \right)^{f_j-f_i} \right],
\end{eqnarray*}
since $\delta_h \leq \delta_{h-r}$ by definition, and thus
$$
\frac{z_j^{f_j}}{z_i^{f_i}}  \in \left[ \left( \eta - \delta_{h-r} \right)^{f_j-f_i}, \left( \eta + \delta_{h-r} \right)^{f_j-f_i} \right] \qquad \text{for $1 \leq i < j \leq h-r$}.
$$
Accordingly, we are in a situation where our induction hypothesis can be applied, and we obtain
$$
\left| c_{1}  z_{1}^{f_{1}}  +  c_2 z_{2}^{f_2} + \dots + c_{h-r}  z_{h-r} ^{f_{h-r}} \right|  \geq \varepsilon_{h-r}.
$$
Continuing from \eqref{continue}, using $e_{r+1} - e_r \geq c$ together with \eqref{etadelta}, \eqref{choice_of_c}, \eqref{continue_2}, and the fact that $y_r^{e_r} \geq 1$,  we have
\begin{eqnarray*}
 \left| b_1 y_1^{e_1} + \dots  + b_h y_h^{e_h} \right|  & \geq & - m r y_r^{e_r} + y_{r+1}^{e_{r+1}} \varepsilon_{h-r}  \\
 & \geq& \left(- m r +  \left( \eta - \delta_h \right)^{e_{r+1} - e_r} \varepsilon_{h-r} \right) y_r^{e_r} \\
  & \geq& \left(- m r +  \left( \eta - \delta_h \right)^c \varepsilon_{h-r}   \right) y_r^{e_r} \\
  & \geq& \left(  - m r +  \left(  \frac{\eta}{2} + \frac{1}{2} \right)^c \varepsilon_{h-r} \right) y_r^{e_r} \\
  & \geq & \left(- m h + 2 m h \right) y_r^{e_r} \\
  & \geq & m h \geq 1.
\end{eqnarray*}
This means that
$$
\left| b_1 y_1^{e_1} + \dots + b_h y_h^{e_h} \right| \geq 1
$$
and we have the desired result in Case 1 (where, judging from the Case 1 analysis alone, we could pick $\varepsilon_h = 1$; however, the actual $\varepsilon_h$ will need to be smaller, following the Case 2 analysis).\\

\item{Case 2:} Among the exponents $e_1 < \dots < e_h$ there does not exists an index $r$ such that $e_{r+1} - e_r \geq c$. This means that $e_h \leq c h$ (recall again that $e_1=0$ by assumption), and thus the class $\mathcal{P}$ contains only finitely many different functions which fall under Case 2. Let us write $\mathcal{P}^*$ for those functions $p \in \mathcal{P}$ for which $e_h \leq ch$. Since $\eta$ is transcendental, we have
$$
p(\eta, \dots, \eta) \neq 0,
$$
and actually, since $\mathcal{P}^*$ is finite,
$$
\min_{p \in \mathcal{P}^*} \left| p(\eta, \dots, \eta)\right| >  0.
$$
Assumption \eqref{y_assumption} implies that, once $\delta_h>0$ is chosen,  we can restrict ourselves to consider only $(y_1, \dots, y_h)$ which satisfy
$$
(y_1, \dots, y_h) \in \left[ \eta - \delta_h, \eta + \delta_h  \right]^{ch}
$$
(recall once more that $e_1=0$). Since the class $\mathcal{P}^*$ is finite, and since $p(y_1, y_2, \dots, y_h)$ depends on $y_1,y_2, \dots, y_h$ in a continuous way, by choosing $\delta_h$ sufficiently small it is possible to ensure that
\begin{equation} \label{pick_delta}
\min_{p \in \mathcal{P}^*} \min_{(y_1, \dots, y_h) \in \left[ \eta - \delta_h, \eta + \delta_h \right]^{ch}} \left| p(y_1, \dots, y_h)\right| >  0.
\end{equation}
Thus we can pick for $\delta_h$ a value for which \eqref{pick_delta} holds true, and for which also $\delta_h \leq \min (\delta_1, \dots, \delta_{h-1})$, so that the requirement from the Case 1 analysis is met. Concerning $\varepsilon$, in the Case 1 analysis it was admissible to choose $\varepsilon_h =1$, so with the choice of
$$
\varepsilon_h := \min \left(1, ~\min_{p \in \mathcal{P}^*} \min_{(y_1, \dots, y_h) \in \left[ \eta - \delta_h, \eta + \delta_h \right]^{ch}} \left| p(y_1, \dots, y_h)\right| \right) > 0
$$
we also cover the Case 2 analysis, and obtain the desired conclusion.

\end{itemize}

\end{proof}

\begin{lemma} \label{lemma1}
Assume that $(a_k)_{k \geq 1}$ satisfies the assumptions of Theorem \ref{th1} and let $m \in\mathbb N$. Then there exists a number $\ell\in\mathbb N$ such that the following holds: if
$$
S_{1,\ell} (\omega) := \sum_{k=1}^\ell \cos(2 \pi a_k \omega)\qquad\text{and}\qquad S_{\ell+1,n} (\omega) := \sum_{k=\ell+1}^n  \cos(2 \pi a_k \omega),
$$
then for all $u,v\in\mathbb N$ with $u+v \leq m$ and for all $n\in\mathbb N$, we have
$$
\mathbb{E} \big[S_{1,\ell}^u S_{\ell+1,n}^v\big] = \mathbb{E} \big[S_{1,\ell}^u\big] \mathbb{E} \big[S_{\ell+1,n}^v\big].
$$
\end{lemma}

\begin{proof}
We recall that 
\begin{equation} \label{gamma}
\lim_{k \to \infty} \frac{a_{k+1}}{a_k} = \eta >1. 
\end{equation}
Let $\delta_1, \dots, \delta_m>0$ and $\varepsilon_1, \dots, \varepsilon_m>0$ be the constants which are provided by Lemma \ref{lemma_poly} when applied for each $h \in\{ 1, \dots, m\}$, and define
\begin{equation} \label{delta_eps_def}
\delta := \min (\delta_1, \dots, \delta_m) \qquad\text{and} \qquad \varepsilon := \min (\varepsilon_1, \dots, \varepsilon_m).
\end{equation}
Then there exists $c:=c(m,\varepsilon)\in\mathbb N$ such that
\begin{equation} \label{c_def}
\gamma^c \geq \frac{2 m}{\varepsilon}, \qquad \text{where} \qquad \gamma := \inf_{k \geq 1} \frac{a_{k+1}}{a_k} > 1.
\end{equation}
Given $\delta$ and $c$, we now choose the number $\ell\in\mathbb N$ in the statement of Lemma \ref{lemma1} so large that
\begin{equation} \label{factor}
\frac{a_{k+1}}{a_k} \in \left[ \eta - \delta, \eta+\delta \right] \qquad \text{for all $k \geq \ell - c m$};
\end{equation}
this is possible since by assumption $\frac{a_{k+1}}{a_k} \to \eta$.

Let $u,v\in\mathbb N$ with $u+v \leq m$ and assume from now on that $n >  \ell$; the statement of Lemma \ref{lemma1} is trivial whenever $n \leq \ell$, since then the sum in the definition of $S_{\ell+1,n}$ is empty.

We have
\begin{eqnarray*}
&&\mathbb{E} \big[S_{1,\ell}^u S_{\ell+1,n}^v\big] \cr
& = & \int_0^1  \left( \sum_{k=1}^\ell \cos(2 \pi a_k \omega) \right)^u \left( \sum_{k=\ell+1}^n  \cos(2 \pi a_k \omega) \right)^v \,d\omega \cr
& = &  \int_0^1 \Bigg( \sum_{1 \leq k_1, \dots, k_u  \leq \ell} \prod_{i=1}^u\cos(2 \pi a_{k_i} \omega) \Bigg)\Bigg( \sum_{\ell+1 \leq k_{u+1} , \dots, k_{u+v} \leq n} \prod_{j=1}^v\cos(2 \pi a_{k_{u+j}} \omega) \Bigg)\,d\omega \cr
& = & \sum_{1 \leq k_1, \dots, k_u  \leq \ell} ~ \sum_{\ell+1 \leq k_{u+1} , \dots, k_{u+v} \leq n} \int_0^1 \prod_{i=1}^u\cos(2 \pi a_{k_i} \omega) \prod_{j=1}^v\cos(2 \pi a_{k_{u+j}} \omega)  \,d\omega
\end{eqnarray*}
Using Euler's formula $\cos(x) = 2^{-1}(e^{ix}+e^{-ix})$, we can establish the product-to-sum identity
 	\[
 		\prod_{i=1}^u\cos(2 \pi a_{k_i} \omega) \prod_{j=1}^v\cos(2 \pi a_{k_{u+j}} \omega) = \frac{1}{2^{u+v}} \sum_{\varepsilon\in\{\pm 1\}^{u+v}} e^{2\pi i \omega \sum_{j=1}^{u+v} \varepsilon_j a_{k_j}},
 	\]
which, together with the orthogonality of the trigonometric system, allows us to express the integral as the counting problem
\begin{eqnarray*}
\mathbb{E} \big[S_{1,\ell}^u S_{\ell+1,n}^v\big] & = & \frac{1}{2^{u+v}}  \sum_{\pm}^{(u+v)} \sum_{1 \leq k_1, \dots, k_u  \leq \ell} ~ \sum_{\ell+1 \leq k_{u+1} , \dots, k_{u+v} \leq n}  \mathbf{1} \left( \pm a_{k_1} \pm a_{k_2} \pm \dots \pm a_{k_{u+v}} = 0 \right);
\end{eqnarray*}
here $\mathbf{1}(\cdot)$ denotes an indicator function and the summation $\sum_{\pm}^{(u+v)}$ has to be understood as a sum over all $2^{u+v}$ many possible configurations of the $u+v$ many $\pm$ signs inside the indicator. Note that some of the indices $k_1, \dots, k_u$ can coincide, and that similarly some of the indices $k_{u+1}, \dots, k_{u+v}$ can coincide. The expression in the previous formula should be compared with
\begin{eqnarray*}
\mathbb{E} \big[S_{1,\ell}^u\big]  \mathbb{E} \big[S_{\ell+1,n}^v\big] & = & \frac{1}{2^{u+v}}  \left( \sum_{\pm}^{(u)} \sum_{1 \leq k_1, \dots, k_u \leq \ell} \mathbf{1} \left( \pm a_{k_1} \pm a_{k_2} \pm \dots \pm a_{k_{u}} = 0 \right) \right) \times \\
& & \qquad \times \left(  \sum_{\pm}^{(v)} \sum_{\ell+1 \leq k_{u+1} , \dots, k_{u+v} \leq n}  \mathbf{1} \left( \pm a_{k_{u+1}} \pm a_{k_{u+2}} \pm \dots \pm a_{k_{u+v}} = 0 \right) \right).
\end{eqnarray*}
A simple comparison between the two expressions shows that we clearly have
	\[
  		\mathbb{E} \big[S_{1,\ell}^u S_{\ell+1,n}^v\big] \geq \mathbb{E} \big[S_{1,\ell}^u\big]  \mathbb{E} \big[S_{\ell+1,n}^v\big],
	\]
and that the potential difference $\mathbb{E} [S_{1,\ell}^u S_{\ell+1,n}^v] - \mathbb{E} [S_{1,\ell}^u]  \mathbb{E} [S_{\ell+1,n}^v]$ would arise from the contribution of ``irreducible'' solutions of $ \pm a_{k_1} \pm a_{k_2} \pm \dots \pm a_{k_{u+v}} = 0$, i.e., such solutions of this equation which cannot be decomposed into a solution of $ \pm a_{k_1} \pm a_{k_2} \pm \dots \pm a_{k_{u}} = 0 $ which is combined with a solution of $\pm a_{k_{u+1}} \pm a_{k_{u+2}} \pm \dots \pm a_{k_{u+v}} = 0$. In other words, to prove that $\mathbb{E} [S_{1,\ell}^u S_{\ell+1,n}^v] = \mathbb{E} [S_{1,\ell}^u]  \mathbb{E} [S_{\ell+1,n}^v]$, we need to establish the following: \\

\emph{Claim: there does not exist any configuration of $\pm$ signs, and any set of indices $(k_1, \dots, k_{u+v})$ with $1 \leq k_1, \dots, k_u \leq \ell$ and $\ell+1 \leq k_{u+1}, \dots, k_{u+v} \leq n$ such that
$$
\pm a_{k_1} \pm \dots \pm a_{k_{u}} \neq 0 \qquad \text{and} \qquad \pm a_{k_{u+1}}  \pm \dots \pm a_{k_{u+v}} \neq 0,
$$
but (with the same configuration of the $\pm$ signs as in the equation above)
$$
 \pm a_{k_1} \pm \dots \pm a_{k_{u+v}} = 0.
$$}

We may assume in the sequel that $k_1, \dots, k_{u+v}$ are sorted in increasing order, i.e., $k_1 \leq k_2 \leq \dots \leq k_{u+v}$. Note again that we cannot rule out the case that some of these indices are equal. We distinguish two cases:\\

\noindent \textbf{Case 1.} There is a ``large gap'' in the index set, i.e., there exist two indices $k_r$ and $k_{r+1}$ for which $k_{r+1} - k_r > c$, where $c$ is the number from \eqref{c_def}.\\

\noindent\textbf{Case 2.} There is no ``large gap'' in the index set, i.e., for all $1 \leq r < u+v$ we have $k_{r+1} - k_r \leq c$.  \\

We shall now prove the claim from above in each of the two cases.\\

 \begin{itemize}
 \item {Proof of claim in Case 1:} To prove the claim, we assume that for a specific configuration of $\pm$ signs and for specific $k_1, \dots, k_{u+v}$ with $1 \leq k_1, \dots, k_u \leq \ell$ and $\ell+1 \leq k_{u+1}, \dots, k_{u+v} \leq n$ we have
$$
\pm a_{k_1} \pm \dots \pm a_{k_{u}} \neq 0 \qquad \text{and} \qquad \pm a_{k_{u+1}}  \pm \dots \pm a_{k_{u+v}} \neq 0.
$$
From this we want to conclude that (with the same choice of $\pm$ signs) we have
$$
 \pm a_{k_1} \pm \dots \pm a_{k_{u+v}} \neq 0.
$$
 Let $r\in\mathbb N$, $1 \leq r < u+v$, be the largest index such that
 	\[
		k_{r+1} - k_r > c.
	\]
This means that for all indices $s > r$,
	\begin{equation}\label{eq:bound on gaps}
		k_{s+1} - k_{s} \leq c.
	\end{equation}
Moreover, we have $k_{u+v} \geq \ell+1$, which is equivalent to
	\[
	 	k_{r+1} \geq \ell+1 - (k_{u+v}-k_{r+1}).
	\]
But using \eqref{eq:bound on gaps} successively,
	\begin{eqnarray*}
		k_{u+v}-k_{r+1} & = & k_{u+v} + \sum_{i=1}^{u+v-(r+2)} \Big( k_{u+v - i} - k_{u+v-i} \Big)  - k_{r+1} \cr
		& = & \sum_{i=0}^{u+v-(r+2)} \Big(k_{u+v-i}-k_{u+v-(i+1)}\Big)\cr
		& \stackrel{\eqref{eq:bound on gaps}}{\leq} & c (u+v-r-1).
	\end{eqnarray*}
This shows that
	\[
		k_{r+1} \geq \ell +1 - c (u+v-r-1) \geq \ell - cm,
	\]
where in the last bound we used that by assumption $u+v \leq m$;	
compare the bound in the previous display with the definition of $\ell$ in \eqref{factor}.

We now study the size of the expression
$$
\left| a_{k_{r+1}} \pm a_{k_{r+2}} \pm \dots \pm a_{k_{u+v}} \right|
$$
for a specific (fixed) configuration of $\pm$ signs. Diving by the smallest term leads to
\begin{equation} \label{can_be_written}
\left| 1 \pm \frac{a_{k_{r+2}}}{a_{k_{r+1}}} \pm \dots \pm \frac{a_{k_{u+v}}}{a_{k_{r+1}}} \right|.
\end{equation}
Now since $k_{r+1} \geq \ell - cm$ as noted above, we can use \eqref{factor}, which yields
\begin{equation} \label{lhs_1}
\frac{a_{k_{r+2}}}{a_{k_{r+1}}} \in \left[ (\eta-\delta)^{k_{r+2} - k_{r+1}} ,(\eta+\delta)^{k_{r+2} - k_{r+1}} \right],
\end{equation}
and similar estimates hold for the quotients $\frac{a_{k_{r+3}}}{a_{k_{r+1}}}$ etc., with the final one being
\begin{equation} \label{lhs_2}
\frac{a_{k_{u+v}}}{a_{k_{r+1}}} \in \left[ (\eta-\delta)^{k_{u+v} - k_{r+1}} ,(\eta+\delta)^{k_{u+v} - k_{r+1}} \right].
\end{equation}
Note that some of the indices $k_{r+1}, k_{r+2}, \dots, k_{u+v}$ might be equal so that the corresponding terms can be combined, such that \eqref{can_be_written} can be rewritten in the form
$$
b_1 y_1^{e_1}+ b_2  y_2^{e_2} + \dots b_h y_h^{e_h}
$$
for some suitable $h \leq u+v \leq m$. Here $b_0, \dots, b_h$ are  suitable coefficients which are all bounded in absolute value by $m$ (since certainly no more than $u+v \leq m$ indices can coincide), $0 = e_1 < e_2 < \dots < e_h$ suitable positive numbers, and $y_1, \dots, y_h$ are suitable numbers which arise as quotients such as those on the left-hand side of \eqref{lhs_1} and \eqref{lhs_2}. The fact that $y_1, \dots, y_h$ arise as such quotients guarantees that \eqref{y_assumption} is satisfied. Note that it is not possible here that $(b_0, \dots, b_h) = (0,\dots,0)$, since this would imply that $\pm a_{k_{u+1}}  \pm \dots \pm a_{k_{u+v}} = 0$, which is ruled out by assumption. Thus we are in a situation where we can apply Lemma \ref{lemma_poly}, which yields that
$$
\left| 1 \pm \frac{a_{k_{r+2}}}{a_{k_{r+1}}} \pm \dots \pm \frac{a_{k_{u+v}}}{a_{k_{r+1}}} \right| \geq \varepsilon,
$$
where $\varepsilon>0$ was defined in \eqref{delta_eps_def}. After multiplying with $a_{k_{r+1}}$, this gives
$$
\left| a_{k_{r+1}} \pm a_{k_{r+2}} \pm \dots \pm a_{k_{u+v}} \right| \geq \varepsilon a_{k_{r+1}}.
$$
On the other hand, we clearly have
$$
\left| \pm a_{k_1} \pm \dots \pm a_{k_r} \right| \leq a_{k_1} + \dots + a_{k_r}  \leq r a_{k_r} \leq m a_{k_r}.
$$
Since $k_{r+1} - k_r \geq c$ by assumption, using the definition of $c$ in \eqref{c_def}, we have
$$
a_{k_{r+1}} \geq \frac{\gamma^c a_{k_r}}{\varepsilon} \geq \frac{2m a_{k_r}}{\varepsilon}.
$$
Thus
\begin{eqnarray*}
\left| \pm a_{k_1} \pm \dots \pm a_{k_{u+v}} \right| & \geq & \left| a_{k_{r+1}} \pm a_{k_{r+2}} \pm \dots \pm a_{k_{u+v}} \right| - \left| \pm a_{k_1} \pm \dots \pm a_{k_r} \right| \\
& \geq & \varepsilon a_{k_{r+1}} - m a_{k_r} \\
& \geq & 2m a_{k_r} - m a_{k_r} > 0.
\end{eqnarray*}
Thus $ \pm a_{k_1} \pm \dots \pm a_{k_{u+v}} \neq 0$, as claimed.\\

\item{Proof of claim in Case 2:} Assume that for all $1 \leq r < u+v$ we have $k_{r+1} - k_r \leq c$. The proof in this case is similar to that of Case 1, but easier. Again our aim is to show that
\begin{equation} \label{eq_to_show}
a_{k_1} \pm a_{k_{2}} \pm \dots \pm a_{k_{u+v}} \neq 0.
\end{equation}
Since $k_{r+1} - k_r \leq c$ for all $r$, we clearly have
$$
k_1 \geq k_{u+v} - c (u+v) \geq \ell - cm.
$$
Thus we are in the regime where \eqref{factor} can be utilized. We can divide by $a_{k_1}$ and study
$$
\left| 1 \pm \frac{a_{k_{2}}}{a_{k_1}} \pm \dots \pm \frac{a_{k_{u+v}}}{a_{k_1}} \right|.
$$
Note, as above, that some of the indices $k_1, \dots, k_{u+v}$ may coincide, so overall we are again led to an expression of the form
$$
b_1 y_1^{e_1}+ b_2  y_2^{e_2} + \dots b_h y_h^{e_h}
$$
for some suitable $h \leq m$, in such a way that the assumptions of Lemma \ref{lemma_poly} are satisfied. Lemma \ref{lemma_poly} yields
$$
\left| 1 \pm \frac{a_{k_{2}}}{a_{k_1}} \pm \dots \pm \frac{a_{k_{u+v}}}{a_{k_1}} \right| \geq \varepsilon,
$$
which establishes \eqref{eq_to_show}.
\end{itemize}
Thus, what we claimed above is indeed true, and accordingly we have
	\[
		\mathbb{E} \big[S_{1,\ell}^u S_{\ell+1,n}^v\big] = \mathbb{E} \big[S_{1,\ell}^u\big] \mathbb{E}\big[S_{\ell+1,n}^v\big].
	\]
This proves the lemma.
\end{proof}

\begin{lemma} \label{lemma2}
Assume that $(a_k)_{k \geq 1}$ satisfies the assumptions of Theorem \ref{th1}. Let $m \geq 1$ be given. Let $\ell$ be the number which was constructed during the proof of Lemma \ref{lemma1}. Let $S_{\ell+1,n}$ be as in the statement of Lemma \ref{lemma1}, and let $\widetilde{S}_{\ell+1,n} = \sum_{k=\ell+1}^n \cos(2 \pi U_k)$, where $U_k\sim\mathrm{Unif}[0,1]$, $\ell+1 \leq k \leq n$, are independent. Then, for all $u \leq m$, we have
$$
\mathbb{E} (S_{\ell+1,n}^u) = \mathbb{E} (\widetilde{S}_{\ell+1,n}^u).
$$
\end{lemma}

\begin{proof}
Lemma \ref{lemma2} can be proved following a strategy very similar to the one which we used to prove Lemma \ref{lemma1}. Using orthogonality of the trigonometric system, it turns out quite quickly that the question whether $\mathbb{E} (S_{\ell+1,n}^u)$ equals $\mathbb{E} (\widetilde{S}_{\ell+1,n}^u)$ or not boils down to the question whether there exists a configuration of $\pm$ signs, together with a system of indices $k_1, \dots, k_u \in \{\ell+1, \dots, n\}^u$, such that
\begin{equation} \label{this_equation}
a_{k_1} \pm \dots \pm a_{k_u} = 0,
\end{equation}
but such that this sum is ``non-trivial'' in the sense that it is not true that all indices show up multiple times, each with the same overall number of ``$+$'' and ``$-$'' signs, and such that the cancellation of the whole expression arises from the cancellation of the partial sums/differences for each particular index. Note here that the ``trivial'' solutions of \eqref{this_equation},  which come from the cancellation of the partial sums/differences for each particular index, are the only ones that also contribute to $\mathbb{E} (\widetilde{S}_{\ell+1,n}^u)$, since by independence there are no further ``non-trivial'' contributions to this expected value. In other words, when combining equal indices and writing
$$
a_{k_1} \pm \dots \pm a_{k_u} = b_1 a_{i_1} + \dots + b_h a_{i_h}
$$
for some $h \leq u$ and distinct $i_1 < \dots < i_h$ and for suitable coefficients $b_1, \dots, b_h$ (all of which are clearly at most $u \leq m$ in absolute value), then by cosidering only ``non-trivial'' solutions we actually rule out the case when the arising set of coefficients is $(b_1, \dots, b_h)$ equals $(0, \dots, 0)$. Now we are in a situation to utilize Lemma \ref{lemma_poly}; note that since we are dealing with $S_{\ell+1,n}$, all indices $k_1, \dots, k_u$ are so large that \eqref{factor} applies. Lemma \ref{lemma_poly} now asserts that no (non-trivial) solutions of \eqref{this_equation} exist (similar as in the proof of Lemma \ref{lemma1}), which proves Lemma \ref{lemma2}.
\end{proof}

\section{Proof of Theorem \ref{th1}}  \label{sec_3}

Let $m \geq 1$ be fixed. Assume that the sequence $(a_k)_{k \geq 1}$ satisfies \eqref{eta_conv} for some transcendental $\eta >1$. Let $\ell$ be the number from the statement of Lemmas \ref{lemma1} and \ref{lemma2}. Let $S_n(\cdot) = \sum_{k=1}^n  \cos (2 \pi a_k \cdot)$, and assume throughout the proof that $n \geq \ell$. In what follows we shall write  $\kappa_m(n):= \kappa_m(S_n)$ for the $m$-th cumulant of $S_n$. Furthermore, we write $\kappa_m^{(1,\ell)}$ and $\kappa_m^{(\ell+1,n)}$ for the $m$-th cumulants of $S_{1,\ell}$ and $S_{\ell+1,n}$, respectively, where $S_{1,\ell}$ and $S_{\ell+1,n}$ are defined as in the statement of Lemma \ref{lemma1}. We claim that
\begin{equation} \label{kappa_claim}
\kappa_m(n)= \kappa_m^{(1,\ell)} + \kappa_m^{(\ell+1,n)}.
\end{equation}
That is, while $S_{1,\ell}$ and $S_{\ell+1,n}$ are not actually stochastically independent, the cumulant of their sum is the sum of cumulants, thus mimicking the behavior of cumulants of sums of  independent random variables. To show that this is indeed the case, consider the cumulants $\overline{\kappa}_m^{(1,\ell)}$ and $\overline{\kappa}_m^{(\ell+1,n)}$ of the random variables
$$
\overline{S}_{1,\ell} :=  \sum_{k=1}^\ell  \cos (2 \pi a_k U) \qquad \text{and} \qquad  \overline{S}_{\ell+1,n} := \sum_{k=\ell+1}^n  \cos (2 \pi a_k V),
$$
where $U, V\sim\mathrm{Unif}[0,1]$ are independent. By construction $\overline{S}_{1,\ell}$ and $\overline{S}_{\ell+1,n}$ are stochastically independent, so that the cumulant $\overline{\kappa}_m(n)$ of $\overline{S}_n := \overline{S}_{1,\ell} +  \overline{S}_{\ell+1,n}$ satisfies
\begin{equation} \label{cumu_addi}
\overline{\kappa}_m(n) = \overline{\kappa}_m^{(1,\ell)} + \overline{\kappa}_m^{(\ell+1,n)}.
\end{equation}
Note that $S_{1,\ell}$ has the same distribution as $\overline{S}_{1,\ell}$, and that $S_{\ell+1,n}$ has the same distribution as $\overline{S}_{\ell+1,n}$. Accordingly,
\begin{equation} \label{kappa_gleich}
\kappa_m^{(1,\ell)}  = \overline{\kappa}_m^{(1,\ell)} \qquad \text{and} \qquad  \kappa_m^{(\ell+1,n)}  =  \overline{\kappa}_m^{(\ell+1,n)}.
\end{equation}
We shall now use the fact that cumulants are fundamentally linked to Bell polynomials in probability theory, where Bell polynomials provide the explicit formulas for converting between a probability distribution's moments and its cumulants \cite[p. 21]{csp2006}. This relation means that for the calculation of the $m$-th cumulant, moments of order up to $m$ are necessary. Thus, the calculation of the cumulants of $\overline{S}_n$ essentially boils down to the calculation of moments of $\overline{S}_n$, so that for some $w \leq m$ one is interested in $\mathbb{E} \left( \overline{S}_n^w \right).$ Using the binomial theorem, one clearly has
\begin{eqnarray*}
\mathbb{E} \left(\overline{S}_n^w \right) & = & \mathbb{E} \left( \left(\overline{S}_{1,\ell} + \overline{S}_{\ell+1, n} \right)^w \right) \\
& = & \sum_{u=0}^w \binom{w}{u} \mathbb{E} \left( \overline{S}_{1,\ell}^u \overline{S}_{\ell+1,n}^{w-u} \right) \\
& = & \sum_{u=0}^w \binom{w}{u} \mathbb{E} \left( \overline{S}_{1,\ell}^u \right) \mathbb{E} \left( \overline{S}_{\ell+1,n}^{w-u} \right),
\end{eqnarray*}
where the last step follows from independence of $ \overline{S}_{1,\ell}$ and $\overline{S}_{\ell+1,n}$. However, for $S_n$ we similarly have
$$
\mathbb{E} \left(S_n^w \right) = \sum_{u=0}^w \binom{w}{u} \mathbb{E} \left( S_{1,\ell}^u S_{\ell+1,n}^{w-u} \right) = \sum_{u=0}^w \binom{w}{u} \mathbb{E} \left( S_{1,\ell}^u \right) \mathbb{E} \left( S_{\ell+1,n}^{v-u} \right),
$$
where the last equality is now not due to independence, but due to our Lemma \ref{lemma1} (which asserts that $S_{1,\ell}$ and $S_{\ell+1,n}$ are ``uncorrelated of higher order'', in the terminology of \cite{ll}; see also \cite{khatri}). Note also that clearly $ \mathbb{E} \left( S_{1,\ell}^u \right) =  \mathbb{E} \left( \overline{S}_{1,\ell}^u \right)$ and $\mathbb{E} \left( S_{\ell+1,n}^{w-u} \right) = \mathbb{E} \left( \overline{S}_{\ell+1,n}^{w-u} \right)$ for all $u$ and $w$. Thus, we have
$$
\mathbb{E} \left(\overline{S}_n^w \right)  = \mathbb{E} \left(S_n^w \right)
$$
for all $w \leq m$. Since $\kappa_m(n)$ and $\overline{\kappa}_m(n)$ are computed by the same combinatorial formula from the moments of $S_n$ and of $\overline{S}_n$, respectively, and since these moments all coincide by our calculation, we have
	\[
  		\kappa_m(n) = \overline{\kappa}_m(n)
 	\]
 This means that by \eqref{cumu_addi} and \eqref{kappa_gleich}, we have
$$
\kappa_m(n) = \overline{\kappa}_m(n) = \overline{\kappa}_m^{(1,\ell)} + \overline{\kappa}_m^{(\ell+1,n)} = \kappa_m^{(1,\ell)} + \kappa_m^{(\ell+1,n)},
$$
which establishes \eqref{kappa_claim}. Now Lemma \ref{lemma2} asserts that all moments up to order $m$ of $S_{\ell+1,n}$ coincide with those of the independent model $\widetilde{S}_{\ell+1,n}$ (as defined in the statement of Lemma \ref{lemma2}). Writing $\widetilde{\kappa}_m^{(\ell+1,n)}$ for the $m$-th cumulant of  $S_{\ell+1,n}$, and using again the fact that cumulants can be calculated in terms of moments, this yields
$$
\kappa_m^{(\ell+1,n)} = \widetilde{\kappa}_m^{(\ell+1,n)} = (n - \ell) \widetilde{\kappa}_m.
$$
Thus, overall we have
$$
\kappa_m(n) = \kappa_m^{(1,\ell)} + \kappa_m^{(\ell+1,n)} = \kappa_m^{(1,\ell)} + (n-\ell) \widetilde{\kappa}_m = n \widetilde{\kappa}_m + \mathcal{O}(1),
$$
as claimed.

\section{Proof of Theorem \ref{th2}}

Throughout this section we write $a_k = 2^k+1$. Let $S_n(\omega) = \sum_{k=1}^n \cos(2 \pi a_k \omega)$, $\omega\in[0,1]$. The calculation of cumulants boils down to the calculation of moments of $S_n$. Trivially, $\kappa_1(n) = \mathbb{E} (S_n) = 0$. Then, by orthogonality, we have
\begin{equation} \label{kappa_2}
\kappa_2(n) = \mathbb{E} (S_n^2) = \frac{n}{2}.
\end{equation}
Next, by orthogonality,
\begin{eqnarray*}
\kappa_3(n) & = & \mathbb{E}(S_n^3) \\
& = & \int_0^1 \sum_{1 \leq k_1,k_2,k_3 \leq n} \cos(2 \pi a_{k_1} \omega)\cos(2 \pi a_{k_2} \omega)\cos(2 \pi a_{k_3} \omega) \\
& = & \frac{1}{8}   \sum_{\pm}^{(3)}  \sum_{1 \leq k_1,k_2,k_3 \leq n} \mathbf{1} \left( \pm a_{k_1} \pm a_{k_2} \pm a_{k_3} = 0 \right),
\end{eqnarray*}
where as in the previous section we write ``$ \sum_{\pm}^{(3)}$'' for the sum over all $8$ possible combinations of $\pm$ signs. However, note that $a_k$ is an odd integer for all $k$, and thus  $\pm a_{k_1} \pm a_{k_2} \pm a_{k_3}$ also always is an odd number. Thus $\pm a_{k_1} \pm a_{k_2} \pm a_{k_3} = 0$ is impossible, and $\kappa_3(n) = 0$ for all $n$. By the same reasoning, all moments of odd order of $S_n$ are zero, and from the way how the cumulants arise out of the moments~\cite[Section~3.2]{peccati_taqqu_book_wiener_chaos}, this implies that $\kappa_5(n)$ and all other cumulants of odd order also vanish, for all $n \geq 1$.\\

For the cumulant of order four, we have $\kappa_4(n) = \mathbb{E} (S_n^4) - 3 \left( \mathbb{E} (S_n^2) \right)^2$. We have
\begin{eqnarray} \label{moment_4}
\mathbb{E} (S_n^4) & = & \frac{1}{16} \sum_{\pm}^{(4)}  \sum_{1 \leq k_1,k_2,k_3,k_4 \leq n} \mathbf{1} \left( \pm a_{k_1} \pm a_{k_2} \pm a_{k_3} \pm a_{k_4} = 0 \right)
\end{eqnarray}
Clearly, $\pm a_{k_1} \pm a_{k_2} \pm a_{k_3} \pm a_{k_4} = 0$ is impossible when all $\pm$ signs are ``$+$''. Similarly, there is no solution when all signs are ``$-$''.\\

Assume now that the first two signs are ``$+$'', and the last two signs are ``$-$''. Then we are looking for solutions of
$$
a_{k_1} + a_{k_2} - a_{k_3} - a_{k_4} = 0 ,
$$
which means
$$
2^{k_1} + 2^{k_2} = 2^{k_3} + 2^{k_4}.
$$
This is clearly always true when
$$
k_1 = k_3 \quad \text{and} \quad k_2 = k_4, \qquad \text{or when} \qquad k_1 = k_4 \quad \text{and} \quad k_2 = k_3,
$$
which happens for $2 n^2 - n$ many configurations of indices. It is easy to see that there are no other solutions, as a consequence of the fact that the binary representation of positive integers is unique (which directly implies the claim that there are no further solutions in the case when $k_1 \neq k_2$; in the case $k_1 = k_2$, it is easily seen that actually one must have $k_1 = k_2 = k_3 = k_4$, which gives a solution that we already took into consideration above).\\

Assume now that the first three signs are ``$+$'' signs, and the last one is a ``$-$'' sign. Thus, we consider the equation
$$
a_{k_1} + a_{k_2} + a_{k_3} - a_{k_4} = 0,
$$
which means
$$
2^{k_1} + 2^{k_2} + 2^{k_3} + 2 = 2^{k_4}
$$
The right-hand side has only one non-zero binary digit. It is easy to see that in order for the left-hand side to have also only one non-zero binary digit, it is necessary for all of $k_1,k_2,k_3,k_4$ to be small. The only solutions $(k_1,k_2,k_3,k_4)$ can easily be found ``by hand'' and turn out to be $(1,1,1,3)$ and $(1,2,3,4)$, where the second solution has to be counted $3!= 6$ times since there are so many possibilities to permute $k_1,k_2,k_3$.\\

\noindent Accordingly, if we assume that $n \geq 4$, then as contributions to $\mathbb{E} (S_n^4)$ we have:
\begin{itemize}
 \item The contribution to \eqref{moment_4} which comes from two ``$+$'' and two ``$-$'' signs. For a fixed configuration of such signs, there are $2n^2-n$ many solutions $(k_1,k_2,k_3,k_4)$. Furthermore, there are $\binom{4}{2}$ possible ways to place two ``$+$'' and two ``$-$'' signs. Thus, the overall contribution of this case is
 $$
 \frac{1}{16} \binom{4}{2} \left(2n^2 -n \right) = \tfrac{3}{4} n^2 - \tfrac{3}{8} n.
 $$
 \item The contribution of the ``sporadic'' solutions $(1,1,1,3)$ and $(1,2,3,4)$, where, as noted, the second solution has to be counted 6 times because of possible permutations. These solutions shows up when there are three ``$+$'' and one ``$-$'' sign, or vice versa. Overall there are $\binom{4}{1} + \binom{4}{3} = 8$ possible ways to place such signs, so the overall contribution arising from these solution is
 $$
\frac{8}{16}   \left(1+6 \right) = \frac{7}{2}.
 $$
\end{itemize}
Combining all this together, we arrive at
\begin{equation} \label{moment_4_size}
\mathbb{E}(S_n^4) =  \tfrac{3}{4} n^2 - \tfrac{3}{8} n + \tfrac{7}{2} \qquad \text{for all $n \geq 4$}.
\end{equation}
Since $\mathbb{E}(S_n^2) = \frac{n}{2}$, this yields
\begin{equation} \label{kappa_4}
\kappa_4(n) = \mathbb{E}(S_n^4) - 3 \left( \mathbb{E} (S_n^2) \right)^2 = -\tfrac{3}{8}n  + \tfrac{7}{2}, \qquad\text{for all $n \geq 4$}.
\end{equation}
Note that this cumulant grows linearly in $n$, and has the same factor $-\tfrac{3}{8}$ which also appears in the independent model; compare \eqref{cumulants_indep}. There is only a slight deviation from $n \widetilde{\kappa}_4$ in the form of the term $\tfrac{7}{2}$ in \eqref{kappa_4}, which arises from some sporadic solutions associated to small indices; accordingly, this term is of a very similar nature to the ``$\mathcal{O}(1)$'' term in the conclusion of Theorem \ref{th1}.\\

Now we come to the calculation of $\kappa_6(n)$, which is given by
\begin{equation} \label{kappa_6}
\kappa_6(n) = \mathbb{E}(S_n^6) - 15 \mathbb{E}(S_n^4) \mathbb{E}(S_n^2) - \underbrace{10 \left(\mathbb{E}(S_n^3)\right)^2}_{=0} + 30 \left( \mathbb{E}(S_n^2) \right)^3.
\end{equation}
As before, orthogonality yields
\begin{eqnarray} \label{moment_6}
\mathbb{E} (S_n^6) & = & \frac{1}{64} \sum_{\pm}^{(6)}  \sum_{1 \leq k_1,k_2,k_3,k_4,k_5,k_6 \leq n} \mathbf{1} \left( \pm a_{k_1} \pm a_{k_2} \pm a_{k_3} \pm a_{k_4} \pm a_{k_5} \pm a_{k_6} = 0 \right).
\end{eqnarray}

To begin with, let us assume that there are three ``$+$'' signs, follows by three ``$-$'' signs. That is, we count solutions of the equation
\begin{equation} \label{six_equ}
a_{k_1} + a_{k_2} + a_{k_3} - a_{k_4} - a_{k_5} - a_{k_6} = 0.
\end{equation}
Clearly there are solutions of the form $k_1 = k_4$, $k_2=k_5$, $k_3=k_6$, and permutations of this. Overall, in the range $k_1,k_2,k_3,k_4,k_5,k_6 \leq n$ there are $6n(n-1)(n-2)$ many such solutions for which $k_1 \neq k_2 \neq k_3$, plus $9 n (n-1)$ many solutions for which two indices among $k_1,k_2,k_3$ are the same but the other is different, plus $n$ many solutions of the form $k_1=k_2=k_3$. Thus, overall we have $6n^3 -9n^2 +4n$ solutions of \eqref{six_equ} for which $k_1 = k_4$, $k_2=k_5$, $k_3=k_6$ or a permuted version of this holds (such that all indices at ``$+$'' signs can be paired with indices at ``$-$'' signs).\\

Assume that two indices at ``$+$'' signs can be paired with two indices at ``$-$'' signs, such as $k_1 = k_4$ and $k_2 = k_5$. Then clearly there can be no solution of \eqref{six_equ} for which $k_3 \neq k_6$. Thus we have already accounted for all such solutions, since for any such solution actually all three indices can be paired. Now assume that exactly one index at a ``$+$'' sign, say $k_1$, can be paired with one index at a ``$-$'' sign, say $k_4$. Such a solutions of \eqref{six_equ} would then require that
$$
a_{k_2} + a_{k_3} - a_{k_5} - a_{k_6} = 0,
$$
where $\{k_2,k_3\} \cap \{k_5,k_6\} = \varnothing$. As explained during the calculation of $\mathbb{E}(S_n^4)$, no such solution exists, essentially due to the uniqueness of the binary representation of integers.\\

We now come to the crucial point in the proof of Theorem \ref{th2}. We will show that the particular structure of our sequence $(a_k)_{k \geq 1} = (2^k+1)_{k \leq 1}$ leads to the existence of quadratically many additional contributions to $\mathbb{E}(S_n^6)$, which result in $\mathbb{E}(S_n^6)$ (and consequently also $\kappa_6(n)$) blowing up in comparison with the independent model. We are now interested in those solutions of \eqref{six_equ} for which
\begin{equation} \label{disjoint_assumption}
\{k_1,k_2,k_3\} \cap \{k_4,k_5,k_6\} = \varnothing.
\end{equation}
Heuristically, what will happen is that by the structure of our particular sequence we have
$$
2 a_k = a_{k+1} + 1 \qquad \text{for all $k$},
$$
and thus $a_k + a_k - a_{k+1}$ always equals $1$. This can be combined with $-a_\ell - a_\ell + a_{\ell+1} = -1$, such that $a_k + a_k - a_{k+1} -a_\ell - a_\ell + a_{\ell+1} = 0$, with quadratically many possible combinations of $k$ and $\ell$. These are the ``additional'' solutions which inflate the size of $\mathbb{E}(S_n^6)$. To do this in a precise way, let us count all solutions $(k_1,k_2,k_3,k_4,k_5,k_6)$ of \eqref{six_equ} such that \eqref{disjoint_assumption} holds. Let us assume that $k_1$ (which has a ``$+$'' sign) is the maximal element among $\{k_1, \dots, k_6\}$, and that $k_1 \geq k_2 \geq k_3$ and $k_4 \geq k_5 \geq k_6$. We are trying to find solutions of
\begin{equation} \label{six_eq_2}
2^{k_1} + 2^{k_2} + 2^{k_3} = 2^{k_4} + 2^{k_5} + 2^{k_6}.
\end{equation}
It is easy to see that when $k_2 = k_1$, then (since $k_1$ was assumed to be maximal, and \eqref{disjoint_assumption} is assumed to hold) we have $2^{k_4} + 2^{k_5} + 2^{k_6} \leq 3 \cdot 2^{k_1-1}$, while the left-hand side is at least $2^{k_1} + 2^{k_2} \geq 2^{k_1 + 1}$, and \eqref{six_eq_2} becomes impossible. Thus actually we must have $k_2 < k_1$, which means that the number on the left-hand side has a binary digit ``1'' at location $k_1$. By \eqref{disjoint_assumption}, the only way for the number on the right-hand side of \eqref{six_eq_2} to have a binary digit ``1'' at location $k_1$ as well, is to choose $k_4 = k_5 = k_1 -1$. Now the remaining variables $k_2,k_3,k_6$ need to satisfy
$$
2^{k_2} + 2^{k_3} = 2^{k_6},
$$
for which the only possibility is that $k_2 + 1 = k_3 + 1 = k_6$. In other words, we have proved that all solutions $(k_1, \dots, k_6)$ respecting \eqref{disjoint_assumption} are of the form $(k+1,\ell,\ell, k,k, \ell+1)$ or some permutation of this, for some $k,\ell$ satisfying $k \neq \ell$ and $1 \leq k,\ell \leq n-1$. Accordingly, the overall number of solutions of \eqref{six_equ}, subject to \eqref{disjoint_assumption} and $k_1, \dots, k_6 \leq n$, is:
\begin{itemize}
 \item  If $|k - \ell| \geq 2$, then $k+1 \neq \ell$ and $k \neq \ell+1$. There are $\binom{3}{2} \binom{3}{1}$ many ways to choose two locations at a ``$+$'' sign  together with one location at a ``$-$'' sign, and $(n-2)(n-3)$ many such pairs of $k$ and $\ell$ with $k,\ell \leq n-1$. This gives a total contribution of $9 (n-2)(n-3) = 9n^2 -45 n + 54$.
\item If $|k-\ell|=1$, then the indices at all ``$+$'' signs, or those at all ``$-$'' signs are interchangeable, so the combinatorial factor is only $3$, and the number of such pairs of $k$ and $\ell$ is $2(n-2)$, giving a total contribution of $6 (n-2)$.
\end{itemize}
Thus overall the solutions satisfying \eqref{disjoint_assumption} give a contribution of $9n^2 -45 n + 54 + 6 (n-2) = 9n^2 -39 n + 42$.\\

Now we come to the case where there are four ``$+$'' and two ``$-$'' signs, and study the equation
\begin{equation} \label{six_equ_4}
a_{k_1} + a_{k_2} + a_{k_3} + a_{k_4} - a_{k_5} - a_{k_6}  = 0,
\end{equation}
which is
\begin{equation} \label{six_four_two}
2^{k_1} + 2^{k_2} + 2^{k_3} + 2^{k_4} + 2 = 2^{k_5} + 2^{k_6}.
\end{equation}
Assume for simplicity of writing in the sequel that $k_1 \geq k_2 \geq k_3 \geq k_4$, and that $k_5 \geq k_6$ (we will consider possible permutations later). If $k_5 = k_6$, then the number on the right-hand side has only one binary digit, and it is easy to see that for the number on the left-hand side of \eqref{six_four_two} to have only one non-zero binary digit, all indices $k_1, \dots, k_4$ must be ``small''. Thus the only solutions in the case $k_5 = k_6$ can be found ``by hand'', and one can check easily that the only such solutions are $(3,1,1,1,3,3)$, $(2,2,2,1,3,3)$ and $(4,3,2,1,4,4)$. \\

Now assume that $k_5 > k_6$. We distinguish several cases:
\begin{itemize}
\item Assume that $k_6 = 1$. Then the equation becomes
$$
2^{k_1} + 2^{k_2} + 2^{k_3} + 2^{k_4} = 2^{k_5}.
$$
It is easy to see that this allows two ``parametric'' solutions: either the solution $(k,k,k,k,k+2,1)$ for $1 \leq k \leq n-2$, or $(k+2,k+1,k,k,k+3,1)$ for $1 \leq k \leq n-3$.
\item Assume that $k_6 = 2$. Then the equation becomes
$$
2^{k_1} + 2^{k_2} + 2^{k_3} + 2^{k_4} = 2^{k_5} + 2.
$$
This is only possible if $k_4=1$, so we get $2^{k_1} + 2^{k_2} + 2^{k_3} = 2^{k_5}$, which gives the parametric solution $(k+1,k,k,1,k+2,2)$ for $1 \leq k \leq n-2$.
\item Assume that $k_6=3$. The equation becomes
$$
2^{k_1} + 2^{k_2} + 2^{k_3} + 2^{k_4} = 2^{k_5} + 6.
$$
This is only possible if $k_4=1$, which gives $2^{k_1} + 2^{k_2} + 2^{k_3} = 2^{k_5} + 4$. There are two ways how this can be true: either $k_2 = k_3 = 1$ and $k_1 = k_5$, or $k_3=2$ and $k_1 + 1 = k_2 + 1 = k_5$. Thus we get two parametric solutions, namely $(k,1,1,1,k,3)$ for $4 \leq k \leq n$, and $(k,k,2,1,k+1,3)$ for $3 \leq k \leq n-1$.
\item Assume that $k_6 = 4$. Then the equation becomes
$$
2^{k_1} + 2^{k_2} + 2^{k_3} + 2^{k_4} = 2^{k_5} + 14.
$$
where $14 = 8 + 4 + 2$ and $k_5  > k_6 = 4$ by assumption. Thus the right-hand side has four non-zero binary digits, and equals the left-hand side for each 6-tuple of the form $(k,3,2,1,k,4)$, where $5 \leq k \leq n$.
\item Assume that $k_6 \geq 5$. In this case, writing \eqref{six_four_two} in the form
$$
2^{k_1} + 2^{k_2} + 2^{k_3} + 2^{k_4} = 2^{k_5} + 2^{k_6} - 2,
$$
the right-hand side has exactly $k_6$ many non-zero binary digits. Since $k_6 \geq 5$, the left-hand side clearly cannot have this many non-zero binary digits, so there are no solutions.
\end{itemize}
~\\

 Finally, we have the case of five ``$+$'' signs and one ``$-$'' sign. Thus we study
 $$
 a_{k_1} + a_{k_2} + a_{k_3} + a_{k_4} + a_{k_5} - a_{k_6} = 0,
 $$
 which is
 $$
 2^{k_1} + 2^{k_2} + 2^{k_3} + 2^{k_4} + 2^{k_5} + 4 = 2^{k_6}.
 $$
 The right-hand side has only one non-zero binary digit. It is easy to see that to ensure that the left-hand also has just one non-zero binary digit, all indices $k_1,\dots, k_5$ need to be ``small''. It is not difficult to find all possible sporadic solutions, namely:
 \begin{eqnarray}
\label{six_sporadic_a} & & (6,5,4,3,2,7),  (5,4,3,1,1,6), (5,4,2,2,2,6), (5,3,3,3,2,6), (4,4,4,3,2,6), \\
\label{six_sporadic_b} & &  (4,2,2,1,1,5),  (3,3,3,1,1,5), (3,3,2,2,2,5), (2,1,1,1,1,4),
 \end{eqnarray}
up to permutations of $k_1,\dots, k_5$, for $n \geq 7$.\\

It remains to factor in the number of ways how a particular solution can arise.
 \begin{itemize}
  \item For the solutions of \eqref{six_equ}, we have $\binom{6}{3}= 20$ possibilities to choose the location of the signs. This has to be multiplied with the term $6n^3 - 9n^2 + 4n + 9n^2 - 39n + 42$ which we got assuming that the signs are located as $(+,+,+,-,-,-)$, giving a total of $120 n^3 - 700 n + 840$.
  \item Then we have the contribution of solutions with four ``$+$'' signs and two ``$-$'' signs, as in \eqref{six_equ_4}, together with the contribution of two ``$-$'' and four ``$+$'' signs. There is an extra factor $\binom{6}{4} + \binom{6}{2} = 30$ from the number of possible ways to choose the location of the signs. Then we have to consider the possible ways of permuting $k_1,k_2,k_3,k_4$ and of permuting $k_5,k_6$ within a solution. Assuming that $n \geq 5$ we have the following:
  \begin{itemize}
  \item The solution $(3,1,1,1,3,3)$ allows $4$ permutations.
  \item The solution $(2,2,2,1,3,3)$ allows $4$ permutations.
  \item The solution $(4,3,2,1,4,4)$ allows $24$ permutations.
  \item The solution $(k,k,k,k,k+2,1)$ allows $2$ permutations, for $1 \leq k \leq n-2$, giving a total contribution of $2 (n-2)$.
  \item The solution $(k+2,k+1,k,k,k+3,1)$ allows $24$ permutations, for $1 \leq k \leq n-3$, giving a total contribution of $24 (n-3)$.
  \item The solution $(k+1,k,k,1,k+2,2)$ allows $8$ permutations if $k=1$, and $24$ permutations, for $2 \leq k \leq n-2$, giving a total contribution of $24 n - 64$.
  \item The solution $(k,1,1,1,k,3)$ allows $8$ permutations, for $4 \leq k \leq n$, giving a total contribution of $8 (n-3)$.
  \item The solutuion $(k,k,2,1,k+1,3)$ allows $24$ permutations, for $3 \leq k \leq n-1$, giving a total contribution of $24 (n-3)$.
  \item The solution $(k,3,2,1,k,4)$ allows $48$ permutations, for $5 \leq k \leq n$, giving a total contribution of $48 (n-4)$.
  \end{itemize}
  Accordingly, the overall contribution from the case of four ``$+$'' signs and two ``$-$'' signs, or vice versa, is $3900 n -11880$, for $n \geq 5$.
  \item Finally, there is the contribution from solutions with five ``$+$'' signs and one ``$-$'' sign, or vice versa. There is a factor $\binom{6}{5} + \binom{6}{1} = 12$ for the number of ways to assign the location of the signs. For each solution listed in \eqref{six_sporadic_a} and \eqref{six_sporadic_b} we need to calculate the number of possible ways to permute $k_1,k_2,k_3,k_4,k_5$. The number of possible permutations for each solution are: $120, 60, 20, 20, 20, 30, 10, 10, 5$. Thus, the overall contribution of these sporadic solutions is 
  \[
  12 \left(120 + 60 + 20 + 20 + 20 + 30 + 10 + 10 + 5 \right) = 3540
  \] 
 for $n \geq 7$.
 \end{itemize}

 As a consequence, by \eqref{moment_6}, we have
 \begin{eqnarray*}
 \mathbb{E} (S_n^6) & = & \frac{1}{64} \left( 120 n^3 - 700 n + 840 +3900 n -11880 + 3540 \right) \\
& = & \frac{30 n^3 + 800 n - 1875}{16}
 \end{eqnarray*}
 for all $n \geq 7$. \\   

Thus, by \eqref{kappa_6} we have
\begin{eqnarray*}
\kappa_6(n) & = &  \frac{30 n^3 + 800 n - 1875}{16} - 15 \left(\frac{3n^2}{4}  - \frac{3n}{8} + \frac{7}{2} \right) \frac{n}{2} + \frac{30n^3}{8} \\
& = & \frac{45n^2 + 380n -1875}{16}
\end{eqnarray*}
for all $n \geq 7$.

\section{A combinatorial formula for cumulants} \label{sec:combinatorial_formula_cumulants}

We fix $(a_k)_{k \geq 1}$, a sequence of natural numbers, and recall that $S_n(\omega) := \sum_{k=1}^n  \cos (2 \pi a_k \omega)$ for $\omega\in[0,1]$.
The aim of this section is to deduce a combinatorial formula for the cumulants of the random variable $S_n$. For $n\in \mathbb N$ and $m\in \mathbb N$,  we define the set
\begin{equation}\label{eq:tuples_nums_signs}
	\mathcal{T}_m(n) :=
	\Big\{
	T =(i_1,\dots,i_m; \varepsilon_1,\dots,\varepsilon_m) \,:\, i_r\in\{1,\dots,n\},\ \varepsilon_r\in\{\pm1\} \text{ for all } r = 1,\ldots,m \Big\}.
\end{equation}
The elements of $\mathcal{T}_m(n)$ will be called \emph{tuples}. We say that $T\in \mathcal{T}_m(n)$ is a \emph{zero-sum tuple} with respect to the sequence $(a_k)_{k \geq 1}$ if and only if $\sum_{r=1}^m \varepsilon_r \, a_{i_r} = 0$.

\subsection{Moments of $S_n$ of order $m$}

We start with the following lemma, expressing the moments of the random variable $S_n$ in terms of the number of zero-sum tuples. Its proof is folklore, and similar to our calculations in Sections \ref{sec_2} and \ref{sec_3} above, but we include it for the sake of completeness.

\begin{lemma}\label{lem:moments}
Let $(a_k)_{k \geq 1}$ be a sequence of natural  numbers. Then, for all $n\in \N,m\in\mathbb N$,
\begin{equation}\label{eq:moment}
\mathbb E[S_n^m] = \frac{1}{2^{m}}\sum_{T= (i_1,\dots,i_m; \varepsilon_1,\dots,\varepsilon_m) \in\mathcal{T}_m(n)}
\mathbf 1_{\sum_{r=1}^m \varepsilon_r \, a_{i_r} = 0}.
\end{equation}
\end{lemma}
\begin{proof}
Expressing  the cosine function in terms of an exponential function, i.e., writing $\cos x=2^{-1}(e^{ix}+e^{-ix})$, we obtain
	\begin{align*}
		\mathbb E [S_n^m] & = \int_{0}^1 \Bigg( \sum_{k=1}^n \cos(2\pi a_k \omega)\Bigg)^m\,d\omega \cr
		& = \frac{1}{2^m}\sum_{k_1=1}^n \ldots \sum_{k_m=1}^n \int_0^1\prod_{r=1}^m \Bigg( e^{2\pi i a_{k_r}\omega} + e^{-2\pi i  a_{k_r}\omega}\Bigg) d\omega \cr
		& = \frac{1}{2^m}\sum_{k_1=1}^n \ldots \sum_{k_m=1}^n \sum_{(\varepsilon_1,\dots,\varepsilon_m)\in\{\pm 1\}^m} \int_0^1  e^{2\pi i \omega \sum_{r=1}^m \varepsilon_ra_{k_r}} d\omega \cr
		& = \frac{1}{2^m}\sum_{k_1=1}^n \ldots \sum_{k_m=1}^n \sum_{(\varepsilon_1,\dots,\varepsilon_m)\in\{\pm 1\}^m} \mathbf 1_{\sum_{r=1}^m \varepsilon_r \, a_{i_r} = 0},
	\end{align*}
which proves the desired identity.	
\end{proof}

\subsection{Cumulants of $S_n$ and multiplicity of tuples}

To establish a combinatorial formula for the cumulants of $S_n$, we need some notation and background on  set partitions. We start by recalling some general terminology and then get back to our specific setting. For more information on set partitions we refer to~\cite[Examples~3.1.1(d), 3.10.4]{stanley_book_enum_comb_1} and~\cite[Section~2.2]{peccati_taqqu_book_wiener_chaos}.

A \emph{partition} of the set $[m]:=\{1,\dots,m\}$ is a collection $\pi = \{A_1,\ldots, A_k\}$ of non-empty, disjoint subsets of $[m]$ such that $A_1\cup \ldots\cup A_k = [m]$. The subsets $A_1,\ldots, A_k$ are called the \emph{blocks} of the partition $\pi$.  The number of blocks in a partition $\pi$ shall be denoted by $|\pi|$. Two partitions differing only by the order of blocks are considered equal.

Given two partitions $\pi, \sigma$ of $[m]$, we say that $\pi$ is \emph{finer} than $\sigma$ if and only if for every block $A\in \pi$ there exists a block  $B\in \sigma$ such that $A \subseteq B$. In this case, we write $\pi \leq \sigma$.  For $m\in \mathbb N$ let $\Pi_m$ be the set of partitions of $[m]$;  note that $\leq$ is a partial order on $\Pi_m$. The maximal element $\hat 1$ is the partition with a single block $[m]$, while the minimal element $\hat 0$ is the partition whose blocks are singletons $\{1\},\ldots, \{m\}$.

Since each set of partitions of $[m]$ has a least upper bound, called their \emph{join}, and a greatest lower bound, called their \emph{meet}, the set $\Pi_m$ forms a lattice. Recall that the meet of two partitions  $\pi, \sigma$, denoted by $\pi \land \sigma$ is the partition whose blocks are the intersections of a block of $\pi$ and a block of $\sigma$, except for the empty set. To define the join $\pi\lor \sigma$ of two partitions $\pi,\sigma\in\Pi_M$, we first form a relation on the blocks $P$ of $\pi$ and $S$ of $\sigma$ by $P\sim S$ if and only if $P$ and $S$ are not disjoint. Then $\pi\lor \sigma$ is the partition in which each block $B$ is the union of a family of blocks connected by this relation.

The Möbius function on the poset $\Pi_m$, denoted by $\mu(\pi, \sigma)$, is known explicitly, see~\cite[Example~3.10.4]{stanley_book_enum_comb_1} or~\cite[Section~2.5]{peccati_taqqu_book_wiener_chaos}. Let us mention  the formula
\[
\mu(\pi, \hat 1) = (-1)^{|\pi|-1}(|\pi|-1)!, \qquad \pi \in \Pi_m.
\]

Recall that $(a_k)_{k\in\mathbb N}$ is a sequence of natural  numbers, fixed once and for all. Consider now a tuple $T =(i_1,\dots,i_m; \varepsilon_1,\dots,\varepsilon_m)\in \mathcal{T}_m(n)$ of indices and signs; see~\eqref{eq:tuples_nums_signs}. For a set $B\subseteq [m]$ define the signed partial sum
\[
\Sigma(B;T):=\sum_{r\in B}\varepsilon_r\,a_{i_r}\in\mathbb{Z};
\]
recall that $T$ is a zero-sum tuple if $\Sigma([m]; T) = 0$. We say that a set $B\subseteq [m]$ is a \emph{$T$-zero-sum set} if and only if $\Sigma(B; T) = 0$. Moreover, we say that a partition $\pi\in \Pi_m$ is a \emph{$T$-zero-sum partition} (with respect to the tuple $T$) if and only if every block of $\pi$ is a $T$-zero-sum set. The family of $T$-zero-sum partitions of $[m]$ will be denoted by
\[
\mathcal U_T  := \Big\{\pi\in\Pi_m \,:\, \Sigma(B; T)= 0 \text{ for every block } B\in \pi \Big\}.
\]

It is clear that the family $\mathcal U_T$  is an upset in $\Pi_m$, which means that whenever $\sigma\in \mathcal U_T$ and $\pi\in\Pi_m$ satisfies $\pi\geq \sigma$, then necessarily $\pi \in \mathcal U_T$. Indeed, if $\pi\geq \sigma$, then every block of $\pi$ is a disjoint union of some blocks of $\sigma$, and since merging disjoint $T$-zero-sum blocks yields a $T$-zero-sum block, we conclude that $\pi \in \mathcal U_T$.

Observe that $\mathcal U_T\neq \varnothing$ if and only if $T$ is a zero-sum tuple. If $\mathcal U_T\neq \varnothing$, then necessarily  $\hat 1 \in \mathcal U_T$, by the upset property. Note also that $\hat 0 \notin \mathcal U_T$ since $a_i\neq 0$ for all $i\in \mathbb N$.

Let $\min(\mathcal U_T)$ be the set of minimal elements of $\mathcal U_T$, with respect to $\leq$. A partition belongs to $\min (\mathcal U_T)$ if and only if it is a $T$-zero-sum partition, but no its proper refinement is a $T$-zero-sum partition.   Now, the \emph{multiplicity} of a tuple $T\in \mathcal{T}_m(n)$ is defined by
$$
{\rm{mult}}(T) := \sum_{\pi \in \mathcal U_T} \mu (\pi, \hat 1).
$$
\begin{example}
If $T$ is not a zero-sum tuple, then $\mathcal U_T = \varnothing$ and ${\rm{mult}}(T) = 0$.
\end{example}

\begin{example}
If $T$ is a zero-sum tuple and $T$ is connected, that is $\Sigma(B;T)\neq 0$ for every
nonempty proper subset $B\subsetneq [m]$, then $\mathcal U_T = \{\hat 1\}$ and ${\rm{mult}}(T) = 1$.
\end{example}
\begin{lemma}\label{lem:multiplicity_combinatorial_formula}
The multiplicity of $T$ is  the alternating count of nonempty subfamilies of $\min(\mathcal U_T)$ whose join is $\hat 1$, more precisely
\begin{equation}\label{eq:lem:multiplicity_combinatorial_formula}
{\rm{mult}}(T) = \sum_{\substack{\varnothing\neq J\subseteq \min(\mathcal U_T)\\ \bigvee J=\hat 1}}(-1)^{|J|+1}.
\end{equation}
\end{lemma}
\begin{proof}
Essentially, the lemma follows from the crosscut theorem; see~\cite[Corollary~3.9.4]{stanley_book_enum_comb_1}. For completeness, we provide a proof. If $\mathcal{U}_T = \varnothing$, then both sides of~\eqref{eq:lem:multiplicity_combinatorial_formula} are $0$. In the following, let $\mathcal{U}_T \neq \varnothing$. Then,
since $\mathcal{U}_T$ is an upset, $\hat 1\in \mathcal{U}_T $ and we have
\[
\mathcal{U}_T=\bigcup_{\tau\in \min(\mathcal{U}_T)} [\tau,\hat 1],
\]
where $[\tau,\hat 1] = \{\pi\in \Pi_m:  \pi \geq \tau\}$.
By the inclusion-exclusion formula for indicator functions, this implies
\[
\mathbf{1}_{\mathcal{U}_T}
=\mathbf{1}_{\bigcup_{\tau\in \min(\mathcal{U}_T)}[\tau,\hat 1]}
=\sum_{\varnothing\neq J\subseteq \min(\mathcal{U}_T)}(-1)^{|J|+1}\,
\mathbf{1}_{\bigcap_{\tau\in J}[\tau,\hat 1]}.
\]
Since $\bigcap_{\tau\in J}[\tau,\hat 1]=\{\pi\in \Pi_m: \pi\ge \tau\ \forall \tau\in J\}
=\{\pi\in \Pi_m:\pi\ge \bigvee J\}=[\bigvee J,\hat 1]$, we get
$$
\mathbf{1}_{\mathcal{U}_T}(\pi)
=\sum_{\varnothing\neq J\subseteq \min(\mathcal{U}_T)}(-1)^{|J|+1}\,
\mathbf{1}_{[\bigvee J,\hat 1]}(\pi),
\qquad \pi \in \Pi_m. 
$$
It follows that 
$$
{\rm{mult}}(T) 
= 
\sum_{\pi\in \mathcal{U}_T}\mu(\pi,\hat 1)
=\sum_{\pi\in \Pi_m}\bm{1}_{\mathcal{U}_T}(\pi) \mu(\pi,\hat 1)
=\sum_{\varnothing\neq J\subseteq \min(\mathcal{U}_T)}(-1)^{|J|+1}
\sum_{\pi\in \Pi_m}\bm{1}_{[\bigvee J,\hat 1]}(\pi) \mu(\pi,\hat 1).
$$
The inner sum restricts to $\pi\ge \bigvee J$, hence
$$
{\rm{mult}}(T)
=\sum_{\varnothing\neq J\subseteq \min(\mathcal{U}_T)}(-1)^{|J|+1}
\sum_{\pi\ge \bigvee J}\mu(\pi,\hat 1)
=
\sum_{\varnothing\neq J\subseteq \min(\mathcal{U}_T)}(-1)^{|J|+1} \bm{1}_{\bigvee J=\hat 1},
$$
where in the last step we applied the identity $\sum_{\pi\ge \sigma}\mu(\pi,\hat 1) = \mathbf{1}_{\sigma = \hat 1}$ with $\sigma=\bigvee J$.
\end{proof}

\begin{example}
Let us show that the multiplicity may take values other than $0$ and $1$. 
Let  $k=4$, and let $b_r:= \varepsilon_r a_{i_r}$, $r=1,\ldots, 4$, be such that  $(b_1,b_2,b_3,b_4)=(1,-1,1,-1)$. The $T$-zero-sum subsets are
$\{1,2\},\ \{3,4\},\ \{1,4\},\ \{2,3\},$ and  $[4]$.
Hence the $T$-zero-sum partitions are
\[
\pi_1=\big\{\{1,2\},\{3,4\}\big\},\qquad
\pi_2=\big\{\{1,4\},\{2,3\}\big\},\qquad
\hat 1=\big\{\{1,2,3,4\}\big\}.
\]
The minimal elements of $\mathcal U_T= \{\pi_1,\pi_2, \hat 1\}$ are $\pi_1$ and $\pi_2$.  Since $\pi_1\vee\pi_2=\hat 1$, the multiplicity equals $-1$.
\end{example}

We are now ready to state a formula for the $m$-th cumulant of the random variable $S_n$.
\begin{lemma}\label{lem:cumulants_combinatorial_formula}
Fix $(a_k)_{k\geq 1} \subseteq \N$, a sequence of natural numbers, and recall that $S_n(\omega) = \sum_{k=1}^n  \cos (2 \pi a_k \omega)$ for $\omega \in [0,1]$. Recall that  $\kappa_m(S_n)$ denotes the $m$-th cumulant of $S_n$, for $m,n\in \mathbb N$. Then
\[
\kappa_m(S_n) \;=\; \frac{1}{2^{m}}\, \sum_{T\in\mathcal{T}_m(n)} {\rm{mult}}(T).
\]
\end{lemma}
\begin{proof}
Let $U\sim\mathrm{Unif}[0,1]$ and  observe that  $S_n$ has the same distribution as $\sum_{k=1}^n \cos(2\pi a_k U) = \frac 12\sum_{k=1}^n (e^{2\pi i a_k U} + e^{-2\pi i a_k U})$.  By multilinearity of cumulants,
\[
\kappa_m(S_n)=\frac{1}{2^m}\sum_{T = (i_1,\dots,i_m; \varepsilon_1,\dots,\varepsilon_m)\in \mathcal{T}_m(n)}\kappa\!\big(e^{2\pi i\varepsilon_1 a_{i_1}U},\dots,e^{2\pi i\varepsilon_m a_{i_m}U}\big),
\]
where $\kappa(\cdot,\dots,\cdot)$ denotes the joint cumulant of the random variables $e^{2\pi i \varepsilon_1 a_{i_1}U}, \dots , e^{2\pi i\varepsilon_m a_{i_m} U}$ (see, e.g., \cite[Section~3.1]{peccati_taqqu_book_wiener_chaos} for its definition).

The classical moment-cumulant formula~\cite[Proposition~3.2.1, Equation~(3.2.7)]{peccati_taqqu_book_wiener_chaos},   expressing the joint cumulant of random variables as an alternate sum of products of their mixed moments, yields
\begin{align*}
\kappa\!\big(e^{2\pi i\varepsilon_1 a_{i_1}U},\dots,e^{2\pi i\varepsilon_m a_{i_m}U}\big)
&=
\sum_{\pi\in\Pi_m}\mu(\pi,\hat1)\prod_{B\in\pi}\mathbb E \left[\prod_{j\in B} e^{2\pi i\varepsilon_j a_{i_j}U}\right]
\\
&=
\sum_{\pi\in\Pi_m}\mu(\pi,\hat1)\prod_{B\in\pi}\mathbb E \left[ e^{2\pi i U \sum _{j\in B} \varepsilon_j a_{i_j}}\right]
\\
&=
\sum_{\pi\in\Pi_m}\mu(\pi,\hat1)\prod_{B\in\pi}\mathbf 1_{\sum _{j\in B} \varepsilon_j a_{i_j} = 0}
\\
&=
\sum_{\pi \in \mathcal U_T} \mu (\pi, \hat 1).
\end{align*}
On the right-hand side we recognize ${\rm{mult}}(T)$, and so the proof is complete.
\end{proof}

\section{Proof of Theorem~\ref{theo:fibonacci_general}} \label{sec:proof_fibonacci_case}
In this section,  we consider lacunary sums $S_n(\omega) = \sum_{k=1}^n  \cos (2 \pi a_k \omega)$, $\omega\in[0,1]$, where $a_1,a_2,\ldots$ are positive integers given by
$$
a_n = c_1 \lambda_1^n + \ldots+ c_d \lambda_d^n, \qquad n\in \mathbb N,
$$
where $d\in \N$ and
\begin{itemize}
\item $\lambda_1,\ldots, \lambda_d\in \mathbb C$ are roots of some irreducible degree $d$ polynomial with integer coefficients;
\item $c_1,\ldots, c_d$ are complex numbers;
\item the following \textit{dominant root condition} holds:
$$
\eta:= \lambda _1 \text{ is real }, \qquad \lambda_1 >1, \qquad \lambda_1 > \max\{|\lambda_2|,\ldots, |\lambda_d|\} =: \rho, \qquad c_1 \neq 0.
$$
\end{itemize}



We prepare the proof of Theorem~\ref{theo:fibonacci_general} with a sequence of lemmas.

\begin{lemma}\label{lem:fibo_sum_upper_bound}
Fix $m\in \N$. There is $L_{1}(m)\in \mathbb N$ such that for all $i_1,\ldots,i_m\in \mathbb N$ we have
$$
|a_{i_1} + \ldots +  a_{i_m}|\leq \eta^{\max\{i_1,\ldots, i_m\} + L_{1}(m)}.
$$
Moreover, we can choose $L_1(m)$ to be increasing in $m$.
\end{lemma}
\begin{proof}
From the formula $a_n = c_1 \lambda_1^n + \ldots+ c_d \lambda_d^n$ and the dominant root condition it follows that $a_n = c_1 \eta^n + \mathcal O(\rho^n)$ as $n\to\infty$, where $\eta >1$ and  $\rho \in (0,\eta)$. Recall also that $a_n\geq 1$ for all $n\in \N$.  Hence,  for suitable constants $0 < c < C$  we have
\begin{equation}\label{eq:fibo_exponential_bounds}
c \eta^n \leq a_n \leq C \eta^n  \quad \text{ for all } \quad n\in \mathbb N.
\end{equation}
Since $\eta >1$, it follows that
$$
|a_{i_1} + \ldots +  a_{i_m}| \leq C (\eta^{i_1} + \ldots + \eta^{i_m}) \leq Cm \eta^{\max\{i_1,\ldots, i_m\}}.
$$
To complete the proof, choose a sufficiently large $L_1(m)$ that fulfills $Cm \leq \eta^{L_1(m)}$.
\end{proof}

In the next two lemmas, we show that a linear relation of the form $\eps_1 a_{i_1} + \ldots +  \eps_m a_{i_m} = 0$ for the terms of the sequence $(a_n)_{n\in \N}$ is \textit{essentially} equivalent to the polynomial relation $\eps_1 \eta^{i_1} + \ldots + \eps_m \eta^{i_m} = 0$ for the dominant root $\eta$. One direction is easy.

\begin{lemma}\label{lem:fibo_relation_lambda_implies_relation_a_i}
Fix $m\in \N$.  For all  $\eps_1,\ldots,\eps_m\in \{\pm 1\}$ and $i_1,\ldots,i_m \in \N$
\begin{equation}\label{eq:relation_lambda_implies_relation_sequence}
\eps_1 \eta^{i_1} + \ldots + \eps_m \eta^{i_m} = 0 \quad \text{ implies } \quad \eps_1 a_{i_1} + \ldots +  \eps_m a_{i_m} = 0.
\end{equation}
\end{lemma}
\begin{proof}
Recall that $\lambda_1,\ldots, \lambda_d$ are roots of an irreducible degree $d$ polynomial over $\mathbb Q$. The Galois group  of the splitting field $K= \mathbb Q(\lambda_1,\ldots, \lambda_d)$ of this polynomial acts transitively on $\lambda_1,\ldots, \lambda_d$. So, $\eta$ can be mapped to any $\lambda_j$ by a suitable  automorphism $\varphi_j$ of the field $K$, and $\eps_1 \eta^{i_1} + \ldots + \eps_m \eta^{i_m} = 0$ implies   $\eps_1 \lambda_j^{i_1} + \ldots + \eps_m \lambda_j^{i_m}= 0$ for every $j=1,\ldots, d$. Taking a linear combination of these identities with coefficients $c_1,\ldots, c_d$ gives $\eps_1 a_{i_1} + \ldots +  \eps_m a_{i_m} = 0$.
\end{proof}

In general, the exact converse of~\eqref{eq:relation_lambda_implies_relation_sequence} need not hold. For example, for the Fibonacci sequence we have $3 a_2 = 3\cdot 2 = 2a_4$ but $3 \eta^2 \neq 2 \eta^3$. As we shall see in Lemmas~\ref{lem:fibo_relation_a_i_implies_relation_lambda} and~\ref{lem:fibo_either_0_or_large_sum}, such  ``sporadic relations'' disappear if all indices $i_1,\ldots,i_m$ are sufficiently large. Moreover, we shall see that  every ``almost cancellation'' among the $a_i$'s must be an exact cancellation.

If $i_1,\ldots, i_m\in \mathbb N$ are integers, we denote by $i_{(1)} \leq \ldots \leq i_{(m)}$ the same numbers sorted in a nondecreasing way. We define the \textit{gap} of the vector $(i_1,\ldots, i_m)$ as
$$
{\mathrm{gap}} (i_1,\ldots, i_m) := \max\{i_{(2)}- i_{(1)},i_{(3)}- i_{(2)},\dots, i_{(m)}- i_{(m-1)}\} \in \N_0.
$$
The next result is a partial converse to Lemma~\ref{lem:fibo_relation_lambda_implies_relation_a_i}.



\begin{lemma}\label{lem:fibo_relation_a_i_implies_relation_lambda}
Fix $m\in \N$ and $\ell\in \mathbb N$. There exist numbers $L_{2}(m,\ell) \in \mathbb N$ and $K_2(m,\ell)\in \N$ such that  the following holds: If $\eps_1,\ldots,\eps_m\in \{\pm 1\}$ and $i_1,\ldots,i_m \geq K_2(m,\ell)$ are integers satisfying ${\mathrm{gap}} (i_1,\ldots, i_m)\leq \ell$, then
$$
\eps_1 a_{i_1} + \ldots +  \eps_m a_{i_m} = 0 \quad \text{ implies }\quad  \eps_1 \eta^{i_1} + \ldots + \eps_m \eta^{i_m} = 0.
$$
Moreover, if $\eps_1 a_{i_1} + \ldots +  \eps_m a_{i_m} \neq 0$, then this number is ``large'' in absolute value in the sense that
\begin{equation}\label{eq:no_cancellation}
|\eps_1 a_{i_1} + \ldots +  \eps_m a_{i_m}| \geq  \eta^{\min \{i_1,\ldots, i_m\} - L_2(m,\ell)}.
\end{equation}
\end{lemma}
\begin{proof}
Since the statement is invariant under permutations of indices, there is no loss of generality in assuming that $i_1\leq \ldots \leq i_m$. Using $a_n= c_1 \lambda_1^n + \ldots+ c_d \lambda_d^n$ and recalling the notation $\eta= \lambda_1$ we write
\begin{equation}\label{eq:proof_no_vanish_small_gaps_tech1}
\eps_1 a_{i_1} + \ldots +  \eps_m a_{i_m} = c_1 (\eps_1 \eta^{i_1} + \ldots + \eps_m \eta^{i_m}) + \sum_{j=2}^d c_j (\eps_1 \lambda_j^{i_1} + \ldots + \eps_m \lambda_j^{i_m}).
\end{equation}
The first term on the right-hand side is the ``principal term''. Recalling that $i_1= \min \{i_1,\ldots, i_m\}$ we write it as
$$
c_1(\eps_1 \eta^{i_1} + \ldots + \eps_m \eta^{i_m}) = c_1 \eta^{i_{1}}  (\eps_1 +\eps_2 \eta^{i_2-i_{1}}  \ldots + \eps_m \eta^{i_m-i_{1}}).
$$
Due to the requirement ${\mathrm{gap}}(i_1,\ldots, i_m)\leq \ell$, there are only finitely many possible values of $|\eps_1 + \eps_2 \eta^{i_2-i_{1}} + \ldots + \eps_m \eta^{i_m-i_{1}}|$. Let $c_0 = c_0(m, \ell) > 0$ be the minimum of these values, ignoring $0$ if it appears in the list of the values. There is a dichotomy between the following two cases.

\vspace*{2mm}
\noindent
\textit{Case 1:} $\eps_1 +\eps_2 \eta^{i_2-i_{1}} + \ldots + \eps_m \eta^{i_m-i_{1}} = 0$. Then $\eps_1 \eta^{i_1} + \ldots + \eps_m \eta^{i_m} = 0$.  By Lemma~\ref{lem:fibo_relation_lambda_implies_relation_a_i}, we conclude that $\eps_1 a_{i_1} + \ldots +  \eps_m a_{i_m} = 0$, and there is nothing to prove.

\vspace*{2mm}
\noindent
\textit{Case 2:}  $\eps_1 +\eps_2 \eta^{i_2-i_{1}} + \ldots + \eps_m \eta^{i_m-i_{1}} \neq 0$. Our aim is to prove the bound~\eqref{eq:no_cancellation}, which implies $\eps_1 a_{i_1} + \ldots +  \eps_m a_{i_m}\neq 0$.  Now, in Case~2, $|\eps_1 +\eps_2 \eta^{i_2-i_{1}}  \ldots + \eps_m \eta^{i_m-i_{1}}| >c_0>0$ and the ``principal term'' in~\eqref{eq:proof_no_vanish_small_gaps_tech1} satisfies
\begin{equation}\label{eq:est_tech_1}
|c_1 (\eps_1 \eta^{i_1} + \ldots + \eps_m \eta^{i_m})| > c_0|c_1| \eta^{i_{1}}.
\end{equation}
To upper-bound the ``remainder'' term in~\eqref{eq:proof_no_vanish_small_gaps_tech1}, recall that $\rho = \max\{|\lambda_2|,\ldots, |\lambda_d|\} < \eta$ and let $\rho_1$ be such that $\max \{\rho, 1\} < \rho_1 < \eta$. Then, for $C'> \max \{|c_2|,\ldots, |c_d|\}$ we have
\begin{equation}\label{eq:est_tech_2}
 \left|\sum_{j=2}^d c_j (\eps_1 \lambda_j^{i_1} + \ldots + \eps_m \lambda_j^{i_m})\right| \leq  C'dm  \rho_1^{i_m}
 \leq
 C'dm \rho_1^{m\ell} \rho_1^{i_{1}} =: C'' \rho_1^{i_{1}}.
\end{equation}
In the last step we used that $i_m \leq i_1 + m\ell$, which follows from ${\mathrm{gap}}(i_1,\ldots, i_m)\leq \ell$. Applying the estimates~\eqref{eq:est_tech_1} and~\eqref{eq:est_tech_2} to the terms appearing in~\eqref{eq:proof_no_vanish_small_gaps_tech1} and using the triangle inequality gives
$$
|\eps_1 a_{i_1} + \ldots +  \eps_m a_{i_m}| \geq  c_0|c_1| \eta^{i_{1}}  - C'' \rho_1^{i_{1}} = \eta^{i_{1}} (c_0|c_1|   - C'' (\rho_1/\eta)^{i_{1}}).
$$
Since $\rho_1/\eta<1$, the number $c_0|c_1|   - C'' (\rho_1/\eta)^{i_{1}}$ is larger than $c_0|c_1|/2$ for sufficiently large $i_1$.  We conclude that if $L_{2}(m,\ell) \in \mathbb N$ and $K_2(m,\ell)\in \N$ are sufficiently large, then $|\eps_1 a_{i_1} + \ldots +  \eps_m a_{i_m}| > \eta^{i_{1} - L_2(m,\ell)}$ for all $i_{1}\geq K_2(m,\ell)$. This completes the proof of~\eqref{eq:no_cancellation}.
\end{proof}

In the next lemma we remove the bounded gaps condition from  Lemma~\ref{lem:fibo_relation_a_i_implies_relation_lambda}.

\begin{lemma}\label{lem:fibo_either_0_or_large_sum}
For every $m\in \mathbb N$ there exist $L_3(m)\in \N$, $K_3(m)\in \N$ such that whenever $\eps_1,\ldots,\eps_m\in \{\pm 1\}$ and $i_1,\ldots,i_m> K_3(m)$ are integers, then
\begin{align*}
\text{ either } \quad
&\eps_1 a_{i_1} + \ldots +  \eps_m a_{i_m} = \eps_1 \eta^{i_1} + \ldots + \eps_m \eta^{i_m} = 0
\\
\text{ or } \quad
&|\eps_1 a_{i_1} + \ldots +  \eps_m a_{i_m}|\geq  \eta^{\min\{i_1,\ldots, i_m\} - L_3(m)}.
\end{align*}
Also, we can choose $L_3(m)$ and $K_3(m)$ to be increasing in $m$.
\end{lemma}
\begin{proof}
In this proof, we always assume that $i_1\leq \ldots\leq i_m$ -- there is no restriction of generality in doing this since the statement is invariant under permutations of indices. We use induction on $m$.

\vspace*{2mm}
\noindent
\textit{Induction base.} For $m=1$, we have $|\eps_1 a_{i_1}| \geq c \eta^{i_1}$, $i_1\in \N$,  by~\eqref{eq:fibo_exponential_bounds}. Choose $L_3(1)$ such that  $c>\eta^{-L_3(1)}$, then  $|\eps_1 a_{i_1}| > \eta^{i_1 - L_3(1)}$ for all $i_1\in \N$.

\vspace*{2mm}
\noindent
\textit{Induction assumption.}
Take some $M\in \N$. Suppose that we already proved the existence of $L_3'(M)= \max_{m=1,\ldots, m} L_3(m) \in \N$ and $K_3'(M)=\max_{m=1,\ldots, m} K_3(m)\in \N$ such that for every $m\in \{1,\ldots, M\}$, every $\eps_1,\ldots,\eps_m\in \{\pm 1\}$ and every integers $i_1,  \ldots, i_m> K_3'(M)$ we have
\begin{equation}\label{eq:tech_induction_assumption}
\begin{aligned}
\text{ either } \quad
&\eps_1 a_{i_1} + \ldots +  \eps_m a_{i_m} = \eps_1 \eta^{i_1} + \ldots + \eps_m \eta^{i_m} = 0
\\
\text{ or }   \quad
&|\eps_1 a_{i_1} + \ldots +  \eps_m a_{i_m}|\geq  \eta^{i_1 - L_3'(M)}.
\end{aligned}
\end{equation}

\vspace*{2mm}
\noindent
\textit{Induction step.}
Consider now some $\eps_1,\ldots,\eps_{M+1}\in \{\pm 1\}$ and some integers $i_1,\ldots,i_{M+1} > K_3(M+1)$, assuming without loss of generality  $i_1\leq \ldots\leq i_{M+1}$. Here, $K_3(M+1)\in \N$ is sufficiently large -- the exact choice will become clear from the argument below. Our aim is to show that
\begin{equation}\label{eq:tech_induction_need_to_prove}
\begin{aligned}
\text{ either } \quad
&\eps_1 a_{i_1} + \ldots +  \eps_{M+1} a_{i_{M+1}} = \eps_1 \eta^{i_1} + \ldots +  \eps_{M+1} \eta^{i_{M+1}} =  0
\\
\quad \text{ or }   \quad
&|\eps_1 a_{i_1} + \ldots +  \eps_{M+1} a_{i_{M+1}}| \geq  \eta^{i_1 - L_3(M+1)},
\end{aligned}
\end{equation}
for a sufficiently large  $L_3(M+1)$ to be chosen below.

\vspace*{2mm}
\noindent
\textit{Case 1:} ${\mathrm{gap}} (i_1,\ldots, i_{M+1}) > L_1(M) +  L_3'(M)$, where $L_1(M)$ comes from Lemma~\ref{lem:fibo_sum_upper_bound}.   This means that $i_j - i_{j-1} > L_1(M) +  L_3'(M)$ for some $j\in \{2,\ldots, M+1\}$. By Lemma~\ref{lem:fibo_sum_upper_bound},
\begin{equation}\label{eq:tech123}
|\eps_1 a_{i_1} + \ldots + \eps_{j-1} a_{i_{j-1}}| \leq \eta^{i_{j-1} + L_{1}(j-1)} \leq  \eta^{i_{j-1} + L_{1}(M)}.
\end{equation}
On the other hand, by the induction assumption applied to $(i_j,\ldots, i_{M+1})$ we have
\begin{equation*}
\begin{aligned}
\text{ either } \quad
&\eps_j a_{i_j} + \ldots + \eps_{M+1} a_{i_{M+1}} = \eps_j \eta^{i_j} + \ldots + \eps_{M+1} \eta^{i_{M+1}} = 0
\\
\text{ or }   \quad
&|\eps_j a_{i_j} + \ldots + \eps_{M+1} a_{i_{M+1}}| \geq  \eta^{i_j - L_3'(M)}.
\end{aligned}
\end{equation*}

\vspace*{2mm}
\noindent
\textit{Case 1a:} If $\eps_j a_{i_j} + \ldots + \eps_{M+1} a_{i_{M+1}} = \eps_j \eta^{i_j} + \ldots + \eps_{M+1} \eta^{i_{M+1}} =0$  then $\eps_1 a_{i_1} + \ldots + \eps_{M+1} a_{i_{M+1}} = \eps_1a_{i_1} + \ldots + \eps_j a_{i_j}$ and $\eps_1 \eta^{i_1} + \ldots + \eps_{M+1} \eta^{i_{M+1}} = \eps_1 \eta^{i_1} + \ldots + \eps_j \eta^{i_j}$ with $j\leq M$. Applying the induction assumption~\eqref{eq:tech_induction_assumption} with $m=j$ to $(i_1,\ldots, i_j)$ gives
\begin{equation*}
\begin{aligned}
\text{ either } \quad
&\eps_1 a_{i_1} + \ldots +  \eps_j a_{i_j} =  \eps_1 \eta^{i_1} + \ldots + \eps_j \eta^{i_j} = 0
\\
\text{ or }   \quad
&|\eps_1 a_{i_1} + \ldots +  \eps_j a_{i_j}|\geq  \eta^{i_1 - L_3'(M)}.
\end{aligned}
\end{equation*}
This gives~\eqref{eq:tech_induction_need_to_prove} provided we choose $L_3(M+1) \geq L_3'(M)$.

\vspace*{2mm}
\noindent
\textit{Case 1b:}
If $|\eps_j a_{i_j} + \ldots + \eps_{M+1} a_{i_{M+1}}| \geq  \eta^{i_j - L_3'(M)}$, the triangle inequality together with~\eqref{eq:tech123}  gives
\begin{align*}
|\eps_1 a_{i_1} + \ldots +  \eps_{M+1} a_{i_{M+1}}|
&\geq
\eta^{i_j - L_3'(M)} - \eta^{i_{j-1} + L_{1}(M)}
\\
&=
\eta^{i_{j-1}+L_{1}(M)}( \eta^{i_j - i_{j-1}- L_3'(M) - L_{1}(M)} - 1)
\\
&\geq
\eta^{i_1+L_{1}(M)}(\eta - 1),
\end{align*}
where in the last step we used $i_{j-1}\geq i_1$ and $i_j - i_{j-1}> L_3'(M) + L_{1}(M)$.  Choosing a sufficiently large $L_3(M+1)$ to ensure that $\eta^{i_{j-1}+L_{1}(M)}(\eta - 1) > \eta^{i_1 - L_3(M+1)}$ completes the proof of~\eqref{eq:tech_induction_need_to_prove}.

\vspace*{2mm}
\noindent
\textit{Case 2:} ${\mathrm{gap}}(i_1,\ldots, i_{M+1}) \leq L_1(M) +  L_3'(M) = : \ell$. Lemma~\ref{lem:fibo_relation_a_i_implies_relation_lambda} with $m= M+1$ yields
\begin{equation*}
\begin{aligned}
\text{ either } \quad
&\eps_1 a_{i_1} + \ldots +  \eps_{M+1} a_{i_{M+1}} = \eps_1 \eta^{i_1} + \ldots +  \eps_{M+1} \eta^{i_{M+1}} = 0
\\
\text{ or }   \quad
&|\eps_1 a_{i_1} + \ldots +  \eps_{M+1} a_{i_{M+1}}| \geq  \eta^{i_1 - L_2(M+1,\ell)}
\end{aligned}
\end{equation*}
provided that $i_1 > K_2 (M+1, \ell)$. If we choose $K_3(M+1) >K_2 (M+1, \ell)$ and $L_3 (M+1) > L_2 (M+1, \ell)$, then~\eqref{eq:tech_induction_need_to_prove} is satisfied. This completes the induction.
\end{proof}

The next lemma states that if a relation between the members of $(a_n)_{n\in \N}$ has a ``large gap'' between the indices, then this relation is reducible: it splits into two relations, one below the gap and one above the gap.
\begin{lemma}\label{lem:fibo_gaps_both_zero}
For every $m\in \mathbb N$ there exists $L_4(m)\in \N$ with the following property:  Whenever $\eps_1,\ldots,\eps_m\in \{\pm 1\}$ and $1\leq i_1\leq \ldots  \leq i_m$ are integers such that $i_{j}- i_{j-1} > L_4(m)$ for some $j\in \{2,\ldots, m\}$, then
$$
\eps_1 a_{i_1} + \ldots +  \eps_m a_{i_m} = 0
\quad \text{ implies } \quad
\eps_1 a_{i_1} + \ldots +  \eps_{j-1} a_{i_{j-1}}
=
\eps_j a_{i_j} + \ldots +  \eps_{m} a_{i_{m}} = 0.
$$
We can choose $L_4(m)$ to be increasing in $m$.
\end{lemma}
\begin{proof}
The essential part of the argument is contained in the proof of Lemma~\ref{lem:fibo_either_0_or_large_sum}. For completeness, we provide the details. Choose $L_4(m)$ such that $L_4(m) > L_3(m) + L_{1}(m)$,  $L_4(m) > K_3(m)$ and $L_4(m) > L_4(j)$ for all $j\le m-1$.  Let $\eps_1 a_{i_1} + \ldots +  \eps_m a_{i_m} = 0$.
On the one hand, by Lemma~\ref{lem:fibo_sum_upper_bound}
\begin{equation}\label{eq:tech123_rep}
|\eps_1 a_{i_1} + \ldots + \eps_{j-1} a_{i_{j-1}}| \leq \eta^{i_{j-1} + L_{1}(j-1)} \leq  \eta^{i_{j-1} + L_{1}(m)}.
\end{equation}
On the other hand, by Lemma~\ref{lem:fibo_either_0_or_large_sum}, applied to the indices $i_{j}, \ldots, i_m$,
$$
\text{ either } \quad
\eps_j a_{i_j} + \ldots + \eps_{m} a_{i_{m}} = 0
\quad \text{ or }   \quad
|\eps_j a_{i_j} + \ldots + \eps_{m} a_{i_{m}}| \geq  \eta^{i_j - L_3(m)}.
$$
(Note that the smallest index satisfies $i_j > L_4(m) > K_3(m-j+1)$.)

\vspace*{2mm}
\noindent
\textit{Case 1a:} If $\eps_j a_{i_j} + \ldots + \eps_{m} a_{i_{m}} = 0$, then also $\eps_j a_{i_j} + \ldots +  \eps_{m} a_{i_{m}}$ and the proof is complete.

\vspace*{2mm}
\noindent
\textit{Case 1b:} If $|\eps_j a_{i_j} + \ldots + \eps_{m} a_{i_{m}}| \geq  \eta^{i_j - L_3(m)}$, then the triangle inequality together with~\eqref{eq:tech123_rep}  gives
\begin{align*}
|\eps_1 a_{i_1} + \ldots +  \eps_{m} a_{i_{m}}|
&\geq
\eta^{i_j - L_3(m)} - \eta^{i_{j-1} + L_{1}(m)}
\\
&=
\eta^{i_{j-1}+L_{1}(m)}( \eta^{i_j - i_{j-1}- L_3(m) - L_{1}(m)} - 1)
\\
&\geq
\eta^{i_1+L_{1}(m)}(\eta - 1) > 0,
\end{align*}
where we used $i_j - i_{j-1}> L_4(m) \geq  L_3(m) + L_{1}(m)+1$. This is a contradiction showing that Case 1b does not occur.
\end{proof}

\begin{proof}[Proof of Theorem~\ref{theo:fibonacci_general}]
Fix $m\in \mathbb N$. By Lemma~\ref{lem:cumulants_combinatorial_formula}, we have
\begin{equation}\label{eq:kappa_as_sum_multiplicities}
\kappa_m(S_n) \;=\; \frac{1}{2^{m}}\, \sum_{T\in\mathcal{T}_m(n)} {\rm{mult}}(T), \qquad n\in \N.
\end{equation}
For a tuple $T = (i_1,\dots,i_m; \varepsilon_1,\dots,\varepsilon_m)\in\mathcal{T}_m(n)$ we write
$$
{\mathrm {gap}} (T) := {\mathrm{gap}} (i_1,\ldots, i_m).
$$

\vspace*{2mm}
\noindent
\textit{Step 1:}
In this step we prove:  Tuples having a gap $> L_4(m)$, where $L_4(m)$ comes from Lemma~\ref{lem:fibo_gaps_both_zero},  have multiplicity $0$ and do not contribute to the sum in~\eqref{eq:kappa_as_sum_multiplicities}. So,
\begin{equation}\label{eq:kappa_as_sum_multiplicities_small_gaps}
\kappa_m(S_n) \;=\; \frac{1}{2^{m}}\, \sum_{T\in\mathcal{T}_m(n)} {\rm{mult}}(T) \bm{1}_{\mathrm{gap}(T) \leq L_4(m)}, \qquad n\in \N.
\end{equation}

 For the proof, we consider a tuple $T\in\mathcal{T}_m(n)$ with $\mathrm{gap}(T) > L_4(m)$. The latter condition means that there is a disjoint decomposition $[m] = J_1\cup J_2$ such that $i_{j_2} > i_{j_1} + L_4(m)$ for all $j_1\in J_1, j_2\in J_2$. We claim:
\begin{itemize}
\item[(a)] For every $T$-zero-sum subset $B\subseteq [m]$, the subsets $B\cap J_1$ and $B\cap J_2$ are also $T$-zero-sum subsets.
\item[(b)] ${\rm{mult}}(T) = 0$.
\end{itemize}

\begin{proof}[Proof of~(a).]
Recall that $B$ is a $T$-zero-sum subset if and only if $\sum_{j\in B} \eps_j a_{i_j} = 0$. If $B\cap J_1= \varnothing$ or $B\cap J_2 = \varnothing$, there is nothing to prove. Suppose that both sets, $B\cap J_1$ and $B\cap J_2$, are non-empty.  By Lemma~\ref{lem:fibo_gaps_both_zero}, applied to the indices $(i_j: j\in B)$, we have $\sum_{j\in B\cap J_1} \eps_j a_{i_j} = \sum_{j\in B\cap J_1} \eps_j a_{i_j} = 0$.
\end{proof}

\begin{proof}[Proof of~(b).]
We use Lemma~\ref{lem:multiplicity_combinatorial_formula}. Recall that $\mathcal U_T$ consists of $T$-zero-sum partitions. If $\pi \in \min (\mathcal U_T)$ is a \textit{minimal} $T$-zero-sum partition, then every block of $\pi$ is contained in $J_1$ or $J_2$. Indeed, if $B$ is some block of $\pi$ not contained in $J_1$ or $J_2$, then splitting $B$ into the blocks $B\cap J_1$ and $B\cap J_2$, which are also $T$-zero-sum sets by (a), yields a $T$-zero-sum partition which is finer than $\pi$ -- a contradiction.  So, every minimal partition in $\mathcal U_T$ is finer than the partition $\{J_1,J_2\}$. Consequently, the  join of minimal elements of $\min (\mathcal U_T)$ cannot be equal to $[m]$.  By Lemma~\ref{lem:multiplicity_combinatorial_formula}, we conclude that ${\rm{mult}}(T) = 0$.
\end{proof}

\vspace*{2mm}
\noindent
\textit{Step 2:} Consider now all tuples $T = (i_1,\dots,i_m; \varepsilon_1,\dots,\varepsilon_m)\in \mathcal{T}_m(n)$ in which all gaps are  $\leq L_4(m)$. Define $K(m) := K_2(m, m L_4(m))$.   The number of such tuples with the additional property that $\min \{i_1,\ldots, i_m\} \leq K(m)$ is finite and for every such tuple $\max\{i_1,\ldots, i_m\} \leq n_0(m)$ for some constant $n_0(m)$. It follows that the number
\begin{equation}\label{eq:tech2_proof_fibo}
B_{m}' := \sum_{T\in\mathcal{T}_m(n)} {\rm{mult}}(T) \bm{1}_{\mathrm{gap}(T) \leq L_4(m), \; \min \{i_1,\ldots, i_m\} \leq K(m)}
\end{equation}
does not depend on  $n \geq n_0(m)$.

\vspace*{2mm}
\noindent
\textit{Step 3:}
It remains to consider tuples $T = (i_1,\dots,i_m; \varepsilon_1,\dots,\varepsilon_m)\in \mathcal{T}_m(n)$ in which all gaps are $\leq L_4(m)$ and  $\min \{i_1,\ldots, i_m\} > K(m)$.  Recall that $i_{(1)} = \min \{i_1,\ldots, i_m\}$ and write $\Delta (i_1,\ldots, i_m) := (i_1 - i_{(1)}, i_2- i_{(1)},\ldots, i_m - i_{(1)}) \in \N_0^m$. Let $\mathcal D_m\subset \N_0^{m}$ be the set of values that $\Delta(i_1,\ldots, i_m)$ can attain for $(i_1,\ldots, i_m) \in \N^m$  with $\mathrm{gap}(i_1,\ldots, i_m) \leq L_4(m)$. Note that the set $\mathcal D_m$ is finite -- this is due to the bound on the size of the  gaps. Take some $\Delta \in \mathcal D_m$, some $\eps' = (\eps_1',\ldots, \eps_m')\in \{\pm1\}^m$ and consider tuples $T\in (i_1,\dots,i_m; \varepsilon_1,\dots,\varepsilon_m)\in \mathcal{T}_m(n)$ with
\begin{equation}\label{eq:fix_param_tuple_T}
\mathrm{gap}(T) \leq L_4(m), \; i_{(1)} > K(m), \; \Delta(i_1,\ldots, i_m)= \Delta, \;
(\eps_1,\ldots, \eps_m) = (\eps_1',\ldots, \eps_m').
\end{equation}

We now claim: The multiplicities of all such tuples are equal to each other. Indeed, by Lemma~\ref{lem:multiplicity_combinatorial_formula}, the multiplicity of $T$ is completely determined by the poset $\mathcal U_T$ of the $T$-zero-sum partitions of $[m]$. Now, for every $B\subseteq [m]$, the maximal gap of the tuple $(i_j: j\in B)$ is at most $m L_4(m)$. Note that $\min \{i_j: j\in B\} \geq i_{(1)}> K(m) = K_2 (m, m L_4(m))$. By Lemmas~\ref{lem:fibo_relation_a_i_implies_relation_lambda} and~\ref{lem:fibo_relation_lambda_implies_relation_a_i}, the set $B$ is a $T$-zero-sum set, i.e.\ $\sum_{j\in B} \eps_{j} a_{i_j} = 0$,  if and only if $\sum_{j\in B} \eps_{j} \eta^{i_j} = 0$.  The latter condition is equivalent to $\sum_{j\in B} \eps_{j} \eta^{i_j-i_{(1)}} = 0$, which depends only on $\Delta(i_1,\ldots, i_m)$ and $(\eps_1,\dots, \eps_m)$. It follows that either $B$ is a $T$-zero-sum set for all $T$ satisfying~\eqref{eq:fix_param_tuple_T}, or $B$ is not a $T$-zero-sum set for all $T$ satisfying~\eqref{eq:fix_param_tuple_T}. Hence, the poset $\mathcal U_T$ is the same for all $T$ satisfying~\eqref{eq:fix_param_tuple_T}.   By Lemma~\ref{lem:multiplicity_combinatorial_formula}, we conclude that all such $T$ have the same multiplicity, and the proof of the claim is complete.

Next we claim: For all sufficiently large $n$, the number of tuples $T\in \mathcal{T}_m(n)$ satisfying~\eqref{eq:fix_param_tuple_T} is given by $n -  c_m(\Delta, \eps')$, where $c_m(\Delta, \eps')$ is a constant depending only on $m, \Delta, \eps'$ (but not on $n$). Indeed, any such tuple $T$ is completely determined by $i_{(1)}$, which is arbitrary satisfying two conditions: (a)  $i_{(1)} > K_1(m)$ and (b)  all entries of $i_{(1)} + \Delta$ are not larger than $n$. This proves the claim.

We can now take the sum over all possible $\Delta \in \mathcal D_m$ and $\eps'\in \{\pm1\}^{m}$.  Recall that both sets are finite and do not depend on $n$. We conclude that, for all sufficiently large $n$, and for suitable integers $w_m$ and $B_m''$, we have
\begin{equation}\label{eq:tech3_proof_fibo}
\sum_{T\in\mathcal{T}_m(n)} {\rm{mult}}(T) \bm{1}_{\mathrm{gap}(T) \leq L_4(m), \; \min \{i_1,\ldots, i_m\} > K(m)} = w_m n + B_m''.
\end{equation}

\vspace*{2mm}
\noindent 
\textit{Step 4:}
Combining~\eqref{eq:kappa_as_sum_multiplicities_small_gaps}, \eqref{eq:tech2_proof_fibo}, \eqref{eq:tech3_proof_fibo}, we conclude that for all sufficiently large $n$, we have
$$
\kappa_m(S_n) = 2^{-m} (w_m n + B_m' + B_m'').
$$
This completes the proof of Theorem~\ref{theo:fibonacci_general}.
\end{proof}

\section*{Acknowledgments}
CA is supported by the Austrian Science Fund (FWF) through projects 10.55776/I4945, 10.55776/I5554, 10.55776/P34763 and 10.55776/P35322. ZK supported by the German Research Foundation (DFG) under Germany's Excellence Strategy EXC 2044/2 -- 390685587, Mathematics M\"unster: Dynamics - Geometry - Structure and by the DFG priority program SPP 2265 Random Geometric Systems. JP is supported by the DFG project 516672205. Part of this work was carried out while ZK and JP were visiting Graz University of Technology. We thank the department for their kind hospitality. We also gratefully acknowledge the support of ChatGPT5. 

\bibliography{Cumulants}
\bibliographystyle{abbrv}

\end{document}